%% file: GFOM_R2_v1.tex
\RequirePackage{fix-cm}
\RequirePackage{amsmath}
\documentclass{svjour3}

\smartqed
\usepackage[hidelinks]{hyperref}
\usepackage{booktabs}
\usepackage{multirow}
\usepackage{lscape}
\usepackage{cancel}
\usepackage{enumerate}
\usepackage{framed}
\usepackage[shortlabels,inline]{enumitem}

\usepackage{amsmath,amsfonts,amssymb} 
\usepackage{graphicx}
\usepackage{color}
\usepackage{subfigure}

\usepackage{tikz}
\usepackage{pgfkeys} 
\usepackage[hmargin=2.5cm,vmargin=3cm,bindingoffset=0.5cm]{geometry}

\usepackage{pdfsync}
\usepackage{epstopdf}
\usepackage{xspace}
\usepackage[normalem]{ulem}
\usepackage[misc]{ifsym}
\usepackage{pgfplotstable}
\pgfplotsset{compat=1.13}

\newcommand{\norm}[1]{{\left\lVert#1\right\rVert}}
\newcommand{\normsq}[1]{{\left\lVert#1\right\rVert^2}}

\DeclareMathOperator{\val}{val}

\DeclareMathOperator{\trace}{Tr}

\DeclareMathOperator*{\argmin}{argmin}

\DeclareMathOperator{\FmuL}{\mathcal{F}_{\mu,L}}
\DeclareMathOperator{\FL}{\mathcal{F}_{0,L}}

\newcommand{\real}{\mathbb{R}}

\newcommand{\bx}{\mathbf{x}}
\newcommand{\bg}{\mathbf{g}}

\newcommand{\bfu}{\mathbf{f}}

\newcommand{\icA}{A^{\mathrm{ic}}}
\newcommand{\ica}{a^{\mathrm{ic}}}
\newcommand{\icb}{b^{\mathrm{ic}}}

\newcommand{\Tr}[1]{{\trace\left(#1\right)}}
\newcommand{\Rd}{\mathbb{R}^d}
\newcommand{\inner}[2]{{\langle #1, #2\rangle}}
\newcommand{\ccp}{c.c.p.\xspace}

\newcommand{\fopt}{f_*}
\newcommand{\xopt}{x_*}

\newcommand{\I}{I^*_N}
\newcommand{\K}{K_N}
\newcommand{\GFOM}{\mathrm{GFOM}}

\newcommand{\AT}[1]{{#1}}
\newcommand{\YD}[1]{{#1}}

\journalname{Journal Name}

\begin{document}
	
\title{Efficient First-order Methods for Convex Minimization: a Constructive Approach\thanks{The second author was supported by the European Research Council (ERC) under the European Union's Horizon 2020 research and innovation program (grant agreement 724063).}}

\titlerunning{A constructive approach to efficient first-order methods}    

\author{Yoel Drori\footnote{\label{fn1}Equal contributions}\addtocounter{footnote}{-1}\addtocounter{Hfootnote}{-1}, Adrien B. Taylor\footnotemark}

\authorrunning{Y.~Drori, A.B.~Taylor}

\institute{Yoel Drori \at Google LLC, 1600 Amphitheatre Parkway, Mountain View, CA 94043, United States\\ E-mail: dyoel@google.com\\ Adrien  Taylor  \at
	INRIA, D\'epartement d'informatique de l'ENS, \'Ecole normale sup\'erieure, CNRS, PSL Research University, Paris, France
	E-mail: adrien.taylor@inria.fr}

\date{Date of current version: \today}

\maketitle

\begin{abstract} 
	We describe a novel constructive technique for devising efficient first-order methods for a wide range of large-scale convex minimization settings, including smooth, non-smooth, and strongly convex minimization.
	The technique \YD{builds upon} a \AT{certain \YD{variant of the} conjugate gradient method} \YD{to construct} a family of methods such that \begin{enumerate*}[a)]
	\item all methods in the family share the same worst-case guarantee as the base \AT{conjugate gradient} method, and
	\item the family includes a fixed-step first-order method.
	\end{enumerate*}
	We demonstrate the effectiveness of the approach by deriving optimal methods for the smooth and non-smooth cases, including new methods that forego knowledge of the problem parameters at the cost of a one-dimensional line search per iteration, and a universal method for the union of these classes that requires a three-dimensional search per iteration.
	In the strongly convex case, we show how numerical tools can be used to perform the construction, and show that the resulting method offers an improved worst-case bound compared to Nesterov's celebrated fast gradient method.
\end{abstract}

	\section{Introduction}
	
	Convex optimization plays a central role in many fields of applications, including optimal control, machine learning and signal processing. In particular, when a large number of variables are involved within a convex optimization problem, the use of first-order methods is more and more widespread due to their typically very attractive low computational cost per iteration. This low computational cost comes, however, at a price: first-order methods often suffer from potentially slow convergence speeds, making them appropriate mostly for obtaining low to medium accuracy solutions.
	Nevertheless, first-order methods remain the methods of choice in many applications and currently receive a lot of attention from the optimization community, which constantly 
	aims at improving them.
	
	An effective and fruitful approach used for analyzing and comparing first-order methods is the study of their worst-case behavior through the black-box model. In this setting, methods are only allowed to gain information on the objective through an oracle, which provides the value and the gradient of the objective at selected points.
	Historically, this approach was largely motivated by the seminal work of Nemirovski and Yudin~\cite{Book:NemirovskyYudin}, 
	and later by the work of Nesterov~\cite{Book:Nesterov}.
	These works established lower and upper bounds on the worst-case performances of first-order methods on several important classes of problems
	and initiated the search for optimal algorithms, which exhibit the best possible worst-case performances (up to a constant factor) for the class of problems they were designed to solve. 

In this work, we consider the generic task of designing first-order methods for convex minimization. 
The suggested approach starts from a conceptual method that does not have an efficient implementation. Then, we show, from the analysis of this method, that one can construct efficient implementations that benefits from the same worst-case guarantees.
This design approach has been considered several times in the past, see \cite{lemarechal1997variable,nemirovski2004prox} and many more. Here, unlike the alluded works, the chosen conceptual method is a very fundamental method capable of handling diverse families of problems, making the design approach applicable to a variety of settings.
The conceptual algorithm we choose is a variant of the conjugate-gradient method, whose analysis therefore occupies an important place in the sequel.

\subsection{Related works}
	
A large number of generic techniques for developing optimization methods were proposed in the past years.
Below we give a short overview of such techniques tailored for convex optimization; we \AT{do} not attempt to give a comprehensive list.
	
	\subsubsection{Links with subspace-search methods} One core idea underlying several classical optimization algorithms, and also strongly related to the technique proposed below, is the use of subspace-searches. Among them,  conjugate gradient methods (see e.g.,~\cite{hestenes1952methods,wright1999numerical}), which can be seen as methods performing minimization steps on increasingly larger subspaces, have a prominent place. 
	
	Related to that, the original optimal methods for smooth convex minimization, developed by Nemirovski and Yudin~\cite{nemirovski1982orth,nemirovski1983information} (see e.g., the review in~\cite{narkiss2005sequential}), requires two and three-dimensional subspace minimization at each iteration and is therefore reminiscent of subspace-search methods. For smooth convex minimization, those methods can be seen as predecessors to the celebrated Nesterov's fast gradient method~\cite{Nesterov:1983wy}, which achieves the same \emph{optimal} convergence rate without relying on those exact subspace minimizations steps. Motivated by similarities between the subspace-search methods of Nemirovski and Yudin~\cite{nemirovski1982orth,nemirovski1983information} and Nesterov's fast gradient method~\cite{Nesterov:1983wy}, we propose a generic technique which transforms subspace-search Conjugate Gradient-like methods to fixed-step methods that have equal or better worst-case performances.
	This technique was also premised in two previous works by the authors:
	\begin{itemize}
		\item In~\cite[Remark 3.1]{drori2016exact}, Drori remarked that the lower bound for smooth convex unconstrained minimization was achieved by a greedy method (referenced to as the \emph{ideal first-order method}),
		\item In~\cite[Section 4.1]{de2016worst}, de Klerk et al.\@{ }study the worst-case complexity of steepest descent with exact line search applied to strongly convex functions. As suggested by an anonymous referee in~\cite{de2016worst}, the worst-case certificates were also valid for the gradient method with an appropriate fixed-step size.
	\end{itemize}
	Links between fast gradient methods and conjugate gradient methods were also recently analyzed in~\cite{karimi2016unified}.
	
	\subsubsection{Links with fast gradient methods}
	Fast gradient schemes for minimizing smooth convex and smooth strongly convex functions originated in the seminal works of Nesterov~\cite{Nesterov:1983wy,Book:Nesterov} (fast gradient methods), Polyak~\cite{polyak1964some,Book:polyak1987} (heavy-ball method) and Nemirovski and Yudin~\cite{nemirovski1982orth,nemirovski1983information}. 
	Despite its fundamental nature, acceleration remained an obscure phenomenon relying on an algebraic trick for years, and many authors have recently developed new explanations for this behavior. For smooth strongly convex minimization, recent popular works include geometrical approaches such as a shrinking ball scheme~\cite{bubeck2015geometric}, a new method based on lower quadratic approximations~\cite{drusvyatskiy2016optimal}, analyses relying on stability theory for discrete-time and/or continuous-time dynamical systems~\cite{fazlyab2017analysis,lessard2014analysis,wilson2016lyapunov} (specifically for fast gradient in~\cite{hu2017dissipativity}), or even as a specific integration scheme of the gradient flow~\cite{scieur2017integration}. In the context of smooth convex minimization, another recent trend include parallels with differential equations~\cite{su2014differential} --- some of the previous approaches for the smooth strongly convex case also apply without strong convexity.
	
\subsubsection{Links with systematic and computer-assisted approaches to worst-case analyses}
	
	This work takes place within the current effort for the development of systematic/computer-guided analyses and design of optimization algorithms. Among them, a systematic approach to lower bounds (which focuses on quadratic cases) is presented in by Arjevani et al.\@~in~\cite{arjevani2016lower}, a systematic use of control theory (via integral quadratic constraints) for developing upper bounds is presented by Lessard et al.\@~in~\cite{lessard2014analysis}, and the performance estimation approach, which aims at finding worst-case bounds was originally developed in~\cite{Article:Drori} (see also surveys in~\cite{drori2014contributions} and~\cite{taylor2017convex}).
	
	Those methodologies are mostly presented as tools for performing worst-cases analyses (see the numerous examples in~\cite{drori2014contributions,hu2017dissipativity,taylor2017convex,taylor2015exact,taylor2015smooth,taylor2017pgm}), however, such techniques were also recently used to develop new methods with improved worst-case complexities. Among others, such an approach was used in \cite{Article:Drori,kim2014optimized} to devise a fixed-step method that attains the best possible worst-case performance for smooth convex minimization \cite{drori2016exact}, and later in \cite{drori2014optimal} to obtain a variant of Kelley's cutting plane method with the best possible worst-case guarantee for non-smooth convex minimization. Also, the control-theoretic approach presented by Lessard et al.\@~in~\cite{lessard2014analysis} was used in~\cite{VanScoy2017} for developing a new accelerated method for smooth strongly convex minimization, called the triple momentum method.

	\subsection{Paper organization and main contributions}
	
	The paper is organized as follows. First, Section~\ref{sec:cvx_interp} introduces elementary facts and definitions that are used throughout this work. Then, Section~\ref{sec:review} presents a specific variant of the conjugate gradient method, which we refer to as the \emph{Greedy First-Order Method} (GFOM), along with the corresponding tools for analyzing it.	
	Following that, Section~\ref{sec:gfom_to_fsfom} proposes a procedure for constructing fixed-step first-order methods that benefits from the same worst-case guarantees as that of GFOM. The procedure is applied on the class of non-smooth and smooth convex functions, and is shown to produce families of first-order methods achieving the best-possible worst-case bounds in both settings.	
	Section~\ref{sec:numerical_gfom_to_fsfom} is devoted the strongly-convex case, where no analytical solution is known to the problem that arises from the design procedure. We show that the resulting numerically-defined algorithm attains both an efficient implementation and an improved worst-case bounds as compared to a standard fast gradient method of Nesterov \cite[Section 2.2]{Book:Nesterov}. Finally, we conclude and discuss extensions in Section~\ref{sec:ccl}.
	
\subsection{Notations}\label{sec:notations}
Consider the convex minimization problem
\begin{equation}
\fopt :=\min_{x\in\Rd} f(x),\label{eq:OptOrig}\tag{OPT}
\end{equation}
with $f\in\mathcal{F}(\mathbb{R}^d)$, for some class $\mathcal{F}(\mathbb{R}^d)$ of closed, convex and proper (\ccp) functions $f:\Rd\rightarrow \mathbb{R}$. For notational convenience, we use the notation $f\in\mathcal{F}$ when the dimension $d$ is left unspecified, the notation $\val\eqref{eq:OptOrig}$ when referring to the optimal value of the problem $\fopt$, and by $\xopt$ to denote an element in $\argmin f$.
	
Additionally, we denote by $x_i\in\Rd$ the iterates produced by the different optimization schemes, and use $f'(x_i)$ to denote an element in the subdifferential $\partial f(x_i)$.
\YD{When $g_i\in \Rd$ is an arbitrary vector, we use the standard notation $g_i\in \partial f(x_i)$ to specify the requirement that $g_i$ is a subdifferential of $f$ at~$x_i$.}
The set $\{x_i\}_{i\in \I}$ containing the first $N$ iterates and an optimal point \AT{is often} used, where the index set $\I$ is defined as follows:
\begin{equation}
\I:=\{*,0,\hdots,N\}.
\end{equation}
\AT{Additionally, we use} the standard notation $\inner{\cdot}{\cdot}:\Rd\times\Rd\rightarrow\real$ to denote the Euclidean inner product, and the corresponding induced norm $\norm{\cdot}$. Given a positive semidefinite matrix $A\succeq 0$, we also use the notation $\inner{\cdot}{\cdot}_A=\inner{A\cdot}{\cdot}$ and the corresponding induced semi-norm $\norm{\cdot}_A$. Finally, we use the notation $(\cdot\odot \cdot):\Rd\times\Rd\rightarrow \real^{d\times d}$ to denote the symmetric outer product, that is, for any $x,y\in\Rd$:
\[ x\odot y=\frac12 (xy^{\top\!}+yx^{\top\!}),\]
resulting in the following useful identity: $\inner{x}{y}_A=\Tr{A(x\odot y)}$.

\section{Basic definitions} \label{sec:cvx_interp}
This section briefly introduces some definitions that 
we \AT{use} in the forthcoming analyses. We start by introducing, for the sake of convenience, a shorthand notation for the set of inputs expected by the algorithms \AT{under consideration}. 
We continue by introducing a notation \AT{allowing} us to implicitly define a function based on some local first-order information.
\begin{definition}\label{def:input}
A pair $(f,x_0)$ is called an \emph{$(\mathcal{F}(\Rd),R_x)$-input} if
$\mathcal{F}(\Rd)$ is a class of closed, convex and proper (\ccp) functions over $\Rd$, $R_x$ is a nonnegative constant, $f\in\mathcal{F}(\Rd)$, $x_0\in \Rd$, and $\norm{x_0-\xopt}\leq R_x$ holds for some $\xopt\in\argmin f(x)$.
\end{definition}
	
\begin{definition}\label{def:Finterpolability}
	Let $\mathcal{F}$ be a class of c.c.p. functions. A set of triplets $S=\{(x_i,g_i,f_i)\}_{i\in I}$ (for some index set $I$) is called \emph{$\mathcal{F}$-interpolable} if there exists a function $f\in\mathcal{F}$ such that $g_i \in \partial f(x_i)$ and $f_i=f(x_i)$ for all $i \in I$.
\end{definition}
	
	For many classes of functions $\mathcal{F}$, the condition ``$S$ is \emph{$\mathcal{F}$-interpolable}'' can be expressed as a finite set of constraints on the elements of $S$. In these cases, we refer to this set of constraints \AT{as} \emph{interpolation conditions} for the class $\mathcal{F}$.
	See~\cite{taylor2015exact,taylor2015smooth} for a list of known interpolation conditions for \AT{different} classes of functions, along with the corresponding proofs.

	For the sake of completeness, we include below the interpolation conditions for the class of strongly-convex smooth functions and for the class of non-smooth convex functions. These classes \AT{are} used in the examples presented in Sections~\ref{sec:gfom_to_fsfom} and~\ref{sec:numerical_gfom_to_fsfom}.

	\begin{theorem}[{\cite[Theorem~4]{taylor2015smooth}}]\label{thm:smooth_sc_ic}
		Let $I$ be a finite index set and let $\FmuL$ denote the set of $L$-smooth and $\mu$-strongly convex functions for some $0\leq \mu < L \leq \infty$.
		A set $\{(x_i,g_i,f_i)\}_{i\in I}$ is $\FmuL$-interpolable if and only if for all $i,j\in I$
		\begin{equation}\label{eq:smooth_sc_ic}
		f_i\geq f_j+\inner{g_j}{x_i-x_j}+\frac{1}{2(1-\mu/L)}\left(\frac1L \norm{g_i-g_j}^2+\mu \norm{x_i-x_j}^2-2\frac{\mu}{L}\inner{x_i-x_j}{g_i-g_j}\right).
		\end{equation}
	\end{theorem}

	\begin{theorem}[{\cite[Theorem~4]{taylor2015exact}}]\label{thm:nonsmooth_ic}
		Let $I$ be a finite index set and let $\mathcal{C}_M$ denote the set of Lipschitz \ccp functions whose gradient is bounded in norm by $M$ for some $0\leq M\leq \infty$.
		A set $\{(x_i,g_i,f_i)\}_{i\in I}$ is $\mathcal{C}_M$-interpolable  if and only if for all $i,j\in I$
		\begin{equation}\label{eq:nonsmooth_ic}
		\begin{aligned}
		&f_i\geq f_j+\inner{g_j}{x_i-x_j},\\ &\norm{g_i} \leq M.
		\end{aligned}
		\end{equation}
	\end{theorem}

	Finally, \YD{we introduce the following technical property, which will be heavily used in establishing tightness results in the sequel.}
	
	\begin{definition} \label{def:contraction} 
		A class of functions $\mathcal{F}(\Rd)$ is said to be \emph{contraction-preserving} if
		for any $\mathcal{F}$-interpolable set $S=\{(x_i,g_i,f_i)\}_{i\in I}$, where $I$ is some finite index set,
		and any $\{\hat x_i\}_{i\in I}\subset \Rd$ satisfying
		\begin{equation}\label{eq:contracted_xi}
		\norm{\hat x_i - \hat x_j}\leq \norm{x_i-x_j},\text{ and } \inner{g_j}{\hat x_i- \hat x_j}=\inner{g_j}{x_i-x_j},\quad \forall i,j \in I,
		\end{equation}
		we have that $\hat S=\{(\hat x_i,g_i,f_i)\}_{i\in I}$ is $\mathcal{F}$-interpolable.
	\end{definition}

	Two important examples of contraction \AT{preserving classes were discussed above}, namely the class of smooth (possibly strongly) convex functions and the class of non-smooth convex functions. 

	\begin{proposition}
		The class of $L$-smooth, $\mu$-strongly convex functions $\mathcal{F}_{\mu,L}$ with $0\leq \mu<L\leq \infty$ is contraction preserving.
		\label{prop:smoothcp}
	\end{proposition}
	\begin{proof}
		Let $I$ be some index set, let $S=\{(x_i,g_i,f_i)\}_{i\in I}$ be a $\mathcal{F}_{\mu,L}$-interpolable set, and suppose $\hat S=\{(\hat x_i,g_i,f_i)\}_{i\in I}$ with $\hat x_i$ satisfying~\eqref{eq:contracted_xi}. Then from Theorem~\ref{thm:smooth_sc_ic}, we have $\forall i,j\in I$:
		\begin{align*}
		f_i&\geq f_j+\inner{g_j}{x_i-x_j}+\frac{1}{2(1-\mu/L)}\left(\frac1L\normsq{g_i-g_j}+\mu\normsq{x_i-x_j}-2\frac\mu L\inner{x_i-x_j}{g_i-g_j}\right),\\
		&\geq f_j+\inner{g_j}{\hat x_i - \hat x_j} + \frac{1}{2(1-\mu/L)}\left(\frac1L\normsq{g_i-g_j}+\mu\normsq{\hat x_i - \hat x_j} - 2\frac\mu L\inner{\hat x_i - \hat x_j}{g_i-g_j}\right),
		\end{align*}
		hence $\hat S$ is $\mathcal{F}_{\mu,L}$-interpolable, as those inequalities are necessary and sufficient for smooth strongly convex interpolation (see Theorem~\ref{thm:smooth_sc_ic}).
	\qed
	\end{proof}
	\begin{proposition}
		The class of \ccp functions with $M$-bounded gradients $\mathcal{C}_M$ with $0\leq M\leq \infty$ is contraction preserving.
	\label{prop:nonsmoothcp}
	\end{proposition}
	\begin{proof}
		Let $S=\{(x_i,g_i,f_i)\}_{i\in I}$ be a $\mathcal{C}_M$-interpolable set and suppose $\hat S=\{(\hat x_i,g_i,f_i)\}_{i\in I}$ with $\hat x_i$ satisfying~\eqref{eq:contracted_xi}. Since the interpolation conditions~\eqref{eq:nonsmooth_ic} hold for $S$, it immediately follows from the equality relations in~\eqref{eq:contracted_xi} that these interpolation conditions also hold for $\hat S$, concluding the proof.
    \qed
	\end{proof}
	
\section{Analysis of a greedy first-order method}\label{sec:review}
	
	The goal of this section is to introduce a framework for studying the worst-case performance of the following subspace-search based greedy method, \AT{reminiscent of conjugate gradient methods}.
	
	\begin{oframed}
		\textbf{Greedy first-order method (GFOM)}
		\begin{itemize}
			\item[] Input: $f\in\mathcal{F}(\mathbb{R}^d)$, initial guess $x_0\in\mathbb{R}^d$, $N\in \mathbb{N}$. \medskip
			\item[] For $i=1,2,\hdots,N$:
			\begin{align}
			&\text{Set } x_{i}\in\ \underset{x\in \mathbb{R}^d}{\mathrm{argmin}} \left\{f(x):\; x\in x_0+\mathrm{span}\{f'(x_0),\hdots,f'(x_{i-1})\}\right\} \label{E:gfomstep}, \\
			&\text{Choose } f'(x_i)\in\partial f(x_i) \text{ such that } \inner{f'(x_i)}{f'(x_j)} = 0 \quad\text{for all } {0\leq j<i}, \label{E:gfom_grad_selection}
			\end{align}    
			\item[] Output: $\GFOM_{N}(f, x_0) := x_N$.
		\end{itemize}  
	\end{oframed}
	
GFOM clearly becomes intractable after the first few iterations; nevertheless, a few fundamental theoretical properties render its analysis interesting independently of the focus of the following sections.
One such property is that GFOM attains the best possible behavior that can be achieved by a first-order method on functions that have a form similar to the ``worst function in the world'' introduced by Nesterov~\cite{Book:Nesterov} (as a way of establishing lower-complexity bounds). GFOM is therefore a natural candidate when looking for ``the best algorithm in the world''.
Additionally, GFOM can be seen as a generalization of the Conjugate Gradient method~\cite{hestenes1952methods} which remains a very fruitful field of study to this day. 
	
	Note that the iteration rule~\eqref{E:gfomstep} is not well-defined when the function $f$ does not attain its infimum on the provided subspace. In such cases, \AT{GFOM} cannot proceed and we say that it does not yield an output.
	For cases where~\eqref{E:gfomstep} is well-defined, note that
	by the first-order optimality conditions, a choice for $f'(x_i)$ that satisfies the (Conjugate Gradient like) conditions~\eqref{E:gfom_grad_selection} necessarily exists, and in particular, when $f$ is differentiable at $x_i$ these requirements are fulfilled by the gradient of $f$ at $x_i$.

\subsection{Estimating the worst-case performance of \AT{GFOM}}\label{sec:peps}

The analysis below is based on the performance estimation methodology which was first introduced in~\cite{Article:Drori} and has been successfully applied to analyze methods in a wide range of settings, including smooth and nonsmooth minimization~\cite{drori2014optimal,kim2014optimized,kim2015convergence,taylor2015smooth}, proximal gradient methods~\cite{taylor2015exact,taylor2017pgm}, saddle-point problems~\cite{drori2014contributions} and more recently to operator splitting methods~\cite{ryu2018operator}.
Here we build upon an approach developed for the analysis of line-searching methods~\cite{de2016worst}, and improve it by providing a tightness proof under some mild conditions.

Clearly, a meaningful analysis can only be attained by making some assumptions on the structure of the problem: namely, that $f$ belongs to some given class of functions~$\mathcal{F}$ and that the initial point $x_0$ satisfies some conditions. In the sequel we restrict our attention to the standard initial condition on $\norm{x_0-\xopt}$, which we assume to be bounded by some constant $R_x>0$ (see Definition~\ref{def:input}). In addition, \AT{we evaluate} the performance of a method in terms of its worst-case absolute inaccuracy $f(x_N)-\fopt$. In other words, we are looking for worst-case guarantees of type
\[ f(x_N)-\fopt\leq \tau \norm{x_0-\xopt}^2,\]
with $\tau\geq 0$ being as small as possible. \YD{Note that the presented analysis allows for more general initial conditions and performance measures (see discussions in~\cite[Section 4]{nemirovsky1992information}, and more specifically in~\cite[Section 4]{taylor2017pgm} in the context of performance \AT{estimation)}, however, for the sake of simplicity we \AT{do not pursue this direction in this work}}.

We start the analysis of \AT{GFOM} with the observation that, under the assumptions discussed above, the worst-case performance of \AT{GFOM} is \emph{by definition} the optimal value to the following \emph{performance estimation problem} (PEP):
	\begin{align}
	\sup_{ f, x_0}&\ f(\GFOM_{N}(f, x_0))-\fopt \tag{PEP}\label{Intro:PEP} \\
	\text{{subject to:} }
	& (f, x_0) \text{ is an } (\mathcal{F}, R_x)\text{-input}. \notag 
	\end{align}
	As an immediate consequence of the definition of~\eqref{Intro:PEP}, if $(f, x_0)$ is such that $f\in\mathcal{F}$ and $\norm{x_0-x_*}\leq R_x$ (for some $x_*\in \argmin_x f(x)$) then the following bound hold:
	\begin{equation}\label{eq:pepbound}
	f(\GFOM_{N}(f, x_0)) - \fopt \leq \val\eqref{Intro:PEP}.
	\end{equation}

Although the form~\eqref{Intro:PEP} appears at first to be a purely theoretical notational reformulation, it provides a convenient framework for manipulating the problem in a way that will eventually yield a tractable bound on the worst-case performance of \AT{GFOM}.

As a first step \AT{for} obtaining tractable bounds we formulate\AT{~\eqref{Intro:PEP}} as a finite-dimensional optimization problem.

\begin{lemma}\label{lem:pep_for_gfom}
Let $N\in \mathbb{N}$, $R_x\geq 0$ and let $\mathcal{F}$ be a class of \ccp functions. 
Then for any $(\mathcal{F}, R_x)$-input $(f, x_0)$,
\begin{equation}\label{eq:gfom_dpep_bound}
	f(\GFOM_{N}(f, x_0)) - \fopt \leq \val \eqref{gfom_dPEP}
\end{equation}
	holds with
\begin{align}
	\sup_{\{(x_i,g_i,f_i)\}_{i\in \I}}&\ f_N-f_* \tag{PEP-GFOM}\label{gfom_dPEP}\\
	\text{{subject to:} }&  \{(x_i,g_i,f_i)\}_{i\in \I} \text{ is }\mathcal{F}\text{-interpolable},   \notag\\
	& \norm{x_0-x_*}\leq R_x, \notag\\
	& g_*=0, \notag \\
	&\inner{g_i}{g_j}= 0, \quad \text{for all } 0\leq j<i=1,\hdots N,\notag\\
	& \inner{g_i}{x_j-x_0}= 0, \quad \text{for all } 1\leq j \leq i=1,\hdots N.\notag
\end{align}
Furthermore, if $\mathcal{F}$ is contraction-preserving (see Definition~\ref{def:contraction}) then the bound~\eqref{eq:gfom_dpep_bound} is tight,
i.e., for any $\varepsilon>0$ there exists an
$(\mathcal{F}, R_x)$-input $(\hat f, \hat x_0)$, such that
\[
	\hat f(\GFOM_{N}(\hat f, \hat x_0)) - \hat f_* \geq \val \eqref{gfom_dPEP}-\varepsilon.
\]
\end{lemma}
	
Problem~\eqref{gfom_dPEP} can be seen as a discretized version of~\eqref{Intro:PEP}, where the \AT{variable $f_i$ acts as the value of the function at $x_i$ and the variable $g_i$ acts as the gradient of the function at $x_i$}. In order to ensure that $x_*$ corresponds to an optimal value of $f$, the constraint $g_*=0$ is included.
In view of this interpretation, all the constraints in~\eqref{gfom_dPEP} are clearly necessary and follow directly from the properties of the problem and from the definition of GFOM, making~\eqref{gfom_dPEP} a relaxation of~\eqref{Intro:PEP} and therefore an upper bound on its value.
The main issue in establishing the tightness claim is to show that the variables $x_i$ fall in the span of the previous gradients (i.e., $x_i \in x_0 + \mathrm{span}\{g_0,\dots,g_{i-1}\}$ as required by~\eqref{E:gfomstep});
here the contraction-preserving assumption is used, allowing to ``project'' the variable $x_i$ on the span of the previous gradients without affecting the constraints or objective, thus showing that an optimal solution can be assumed to satisfy the required property~\eqref{E:gfomstep}. We postpone the formal proof of Lemma~\ref{lem:pep_for_gfom} to Appendix~\ref{sec:GFOMproof}.
	
In surprisingly many situations, the constraint `$\{(x_i,g_i,f_i)\}_{i\in \I} \text{ is } \mathcal{F}\text{-interpolable}$'  in~\eqref{gfom_dPEP} can be expressed as a finite set of inequalities that depend on the $\{(x_i,g_i,f_i)\}_{i\in \I}$ variables via \AT{linear combinations of $\{f_i\}_{i\in \I}$ and inner products of the vectors $\{(x_i, g_i)\}_{i\in \I}$}
	(for example, see~\eqref{eq:smooth_sc_ic} and~\eqref{eq:nonsmooth_ic} for the $\FmuL$ and $\mathcal{C}_M$ classes, respectively). In these cases, \eqref{gfom_dPEP} becomes a quadratically constrained quadratic program (QCQP), which has efficient SDP relaxations~\cite{beck2007quadratic,beck2012new}.
	We now consider such cases, and provide sufficient conditions under which these relaxations are exact.

	\subsection{A tractable bound on the worst-case performance of \AT{GFOM}}\label{sec:primalpep}
	
We begin by introducing the following notations:
given a set of triplets $S=\{(x_i,g_i,f_i)\}_{i\in \I}$ with $g_*=0$, let $P\in\mathbb{R}^{d\times (2N+2) }$ and $F\in\mathbb{R}^{N+2}$ be defined as containing the information collected after $N$ iterations as follows:
\begin{equation}\label{eq:coord_and_grads_gfom}
	\begin{array}{l}
	P=[\ g_0 \ g_1 \ \hdots \ g_N \ | \ x_1-x_0 \ \hdots \ x_N-x_0 \ | \ x_*-x_0\ ], \\
	F=[\ f_0 \ f_1 \ \hdots \ f_N \  f_* \ ]^\top.
	\end{array}
\end{equation}
	Further denote by $G \in \mathbb{R}^{(2N+2)\times (2N+2)}$ the corresponding positive semidefinite Gram matrix 
	\[
	G=P^{\top\!}P\succeq 0, 
	\]
	and by $\bx_i, \bg_i \in \mathbb{R}^{2N+2}$ and $\bfu_i\in \mathbb{R}^{N+2}$ the following zero and unit vectors
	\begin{alignat*}{2}
	\bx_0 &:= 0, \\
	\bx_i &:= e_{N+1+i}, \quad &&i=1,\dots,N,\\
	\bx_* &:= e_{2N+2}, \\
	\bg_* &:= 0, \\
	\bg_i &:= e_{i+1}, &&i=0,\dots,N,\\
	\bfu_i &:= e_{i+1}, &&i=0,\dots,N,\\
	\bfu_* &:= e_{N+2},
	\end{alignat*}
	which are defined such that $x_i-x_0=P(\bx_i-\bx_0)$, $g_i=P\bg_i$ and $f_i=F \bfu_i$.
	Using these notations together with the notations defined in Section~\ref{sec:notations}, the following reformulations hold:
	\begin{equation}\label{eq:sdp_vardef}
	\begin{aligned}
	\inner{g_i}{g_j}&=\inner{\bg_i}{\bg_j}_G= 0,\\
	\inner{g_i}{x_j-x_0}&=\inner{\bg_i}{\bx_j-\bx_0}_G= 0,\\
	\norm{x_0-x_*}^2&=\norm{\bx_0-\bx_*}_G^2.
	\end{aligned}
	\end{equation}

These notations allow us to encode the equality and inequality constraints in~\eqref{gfom_dPEP} within an SDP.
In order to encode the interpolation conditions, we \AT{further require the} class $\mathcal{F}(\Rd)$ \AT{to have} interpolation conditions that can be expressed as a set of affine constraints in the entries of the matrices $G$ and $F$ defined above.
\begin{assumption}\label{a:ic}
    The constraint `$\{(x_i,g_i,f_i)\}_{i\in \I} \text{ is } \mathcal{F}(\Rd)\text{-interpolable}$' can be encoded within an SDP. 
    I.e., there exists an index set $\K$ and an appropriate choice of matrices $\icA_k\in\mathbb{R}^{(2N+2)\times (2N+2)}$, vectors $\ica_k\in\mathbb{R}^{N+2}$ and scalars $\icb_k\in\mathbb{R}$ for all $k\in \K$, such that
	for any set of triples $S=\{(x_i,g_i,f_i)\}_{i\in \I}$,
	$S$ is $\mathcal{F}(\Rd)$-interpolable \emph{if and only if}
	\begin{equation}\label{E:icmatrixform}
	    \Tr {\icA_kG}+(\ica_k)^\top F+\icb_k\leq 0,  \quad\text{for all } k\in \K.
	\end{equation}
	 (The notation \emph{ic} above is an abbreviation of \emph{interpolation conditions}.)
\end{assumption}

	\begin{example}[Interpolation conditions for $\FmuL$]\label{ex:strongly_convex_ic}
		For the class of $L$-smooth and $\mu$-strongly convex functions, the index set $\K$ can be defined by $\K=\{(i,j):i,j\in \I\}$, and $\{(\icA_{k}, \ica_k, \icb_k)\}_{k\in \K}$ can be defined by:
		\begin{align*}
		\icA_{(i,j)}&=\bg_j\odot (\bx_i-\bx_j) + \frac{1}{2(1-\mu/L)} \biggl(\frac{1}{L}(\bg_i-\bg_j)\odot(\bg_i-\bg_j) \\&\qquad + \mu (\bx_i-\bx_j)\odot(\bx_i-\bx_j)-2\frac{\mu}{L} (\bx_j-\bx_i)\odot(\bg_j-\bg_i) \biggr),\\
		\ica_{(i,j)}&=\bfu_j-\bfu_i, \\
		\icb_{(i,j)} &= 0,
		\end{align*} 
		for all $(i,j)\in \K$ (see Theorem~\ref{thm:smooth_sc_ic}).
	\end{example}
	
	\begin{example}[Interpolation conditions for $\mathcal{C}_M$]\label{ex:nonsmooth_ic}
		Consider the class of Lipschitz \ccp functions whose gradients are bounded in norm by $M$ for some $0\leq M\leq \infty$.
		In this setting, there are two types of constraints encoding the interpolation conditions (see Theorem~\ref{thm:nonsmooth_ic}): constraints bounding the gradients, and constraints ensuring convexity.
		We therefore set
		$\K=\I \cup \{(i,j)\in \I\times \I: i\neq j\}$
		and define $\{(\icA_{k}, \ica_k, \icb_k)\}_{k\in \K}$ as follows:
		\begin{alignat*}{2}
		& \icA_i = \bg_i\odot \bg_i, \quad && i\in \I,\\
		& \ica_i = 0, \quad && i\in \I, \\
		& \icb_i = -M^2, \quad && i\in \I, \\
		& \icA_{(i,j)}= \bg_j\odot (\bx_i-\bx_j), \quad && i\neq j\in \I, \\
		& \ica_{(i,j)}= \bfu_j-\bfu_i, \quad && i\neq j\in \I,\\
		& \icb_{(i,j)}= 0, \quad && i\neq j\in \I.
		\end{alignat*}
	\end{example}
	
	We can now formulate a tractable performance estimation problem for \AT{GFOM}.
\begin{lemma}\label{lem:pep_sdp_relax}
Let $(f, x_0)$ be an $(\mathcal{F}(\Rd), R_x)$-input, $N\in \mathbb{N}$, and suppose $\{(\icA_{k}, \ica_k, \icb_k)\}_{k\in \K}$ encodes the interpolation conditions for $\mathcal{F}(\Rd)$ (see Assumption~\ref{a:ic}). Then
\begin{equation}\label{L:sdpBound}
	\val \eqref{gfom_dPEP} \leq \val\eqref{gfom_sdpPEP},
\end{equation}
where
\begin{align}
	&\sup_{ F\in\mathbb{R}^{N+1}, G\in\mathbb{R}^{2N+2\times 2N+2}} F^\top \bfu_N - F^\top \bfu_*  \tag{sdp-PEP-GFOM}\label{gfom_sdpPEP} \\
	& \begin{array}{lrl}
	\text{{subject to:} } 
	&\Tr {\icA_kG}+(\ica_k)^\top F+\icb_k\leq 0, & \quad\text{for all } k\in \K,\\
	&\inner{\bg_i}{\bg_j}_G= 0, &\quad\text{for all } 0\leq j<i=1,\hdots N,\\
	&\inner{\bg_i}{\bx_j-\bx_0}_G= 0, &\quad\text{for all } 1\leq j \leq i=1,\hdots N,\\
	&\norm{\bx_0-\bx_*}_G^2-R_x^2\leq 0, &\\
	& G\succeq 0.
	\end{array}\notag
\end{align}
Furthermore, if $d \geq 2N+2$ then~\eqref{L:sdpBound} holds with equality.
\end{lemma}
\begin{proof}
	Since any feasible solution $\{(x_i, g_i, f_i)\}_{i\in \I}$ to~\eqref{gfom_dPEP}  can be transformed to a feasible solution to~\eqref{gfom_sdpPEP}, by setting $G=P^{\top\!}P$, where $P$ is defined as in~\eqref{eq:coord_and_grads_gfom}, then~\eqref{L:sdpBound} immediately follows.
	
	Now, suppose $d \geq 2N+2$ and let $(F, G)$ be feasible to~\eqref{gfom_sdpPEP}. 
	Since $G$ is a $(2N+2)\times (2N+2)$ positive-semidefinite matrix, there exits some $d \times (2N+2)$ matrix $P$ such that $G=P^{\top\!}P$,
	and thus $(F, G)$ can be transformed to a feasible solution for~\eqref{gfom_dPEP} by assigning values for \AT{$\{(x_i, g_i, f_i)\}_{i\in \I}$} according to~\eqref{eq:coord_and_grads_gfom}. We have obtained
	\[
		\val\eqref{gfom_sdpPEP} \leq \val \eqref{gfom_dPEP},
	\]
	which completes the proof.
	\qed
\end{proof}

We now describe the final form of the bound on the performance of \AT{GFOM}, which is the standard Lagrangian dual of~\eqref{gfom_sdpPEP}. 
This bound \AT{provides} the basic building block for SSEP.

\begin{theorem}\label{thm:dual_pep_for_gfom}
Let $N\in \mathbb{N}$, $R_x\geq 0$, $\mathcal{F}(\Rd)$ a class of \ccp functions and suppose $\{(\icA_{k}, \ica_k, \icb_k)\}_{k\in \K}$ encodes the interpolation conditions for $\mathcal{F}$ (see Assumption~\ref{a:ic}). 
Then for any $(\mathcal{F}(\Rd), R_x)$-input $(f, x_0)$,
\begin{equation}\label{eq:thm1}
	f(\GFOM_{N}(f, x_0)) - \fopt \leq \val \eqref{gfom_dualPEP}
\end{equation}
holds with
\begin{align*}
	\inf_{\{\alpha_k\},\{\beta_{i,j}\},\{\gamma_{i,j}\},\tau_x} &\ \tau_x R_x^2-\sum_{k\in \K} \alpha_k \icb_k \tag{dual-PEP-GFOM}\label{gfom_dualPEP}\\
	\text{subject to: }&\alpha_k\geq 0,\ \tau_x\geq 0,\notag\\
	&\begin{aligned}
	\sum_{k\in \K}\alpha_k \icA_k&+\sum_{i=1}^N \bg_i\odot \left[\sum_{j=0}^{i-1}\beta_{i,j}\bg_j+\sum_{j=1}^i\gamma_{i,j}(\bx_j-\bx_0) \right]\\&+
	\tau_x[(\bx_0-\bx_*)\odot(\bx_0-\bx_*)]\succeq 0,
	\end{aligned}\notag \\
	& \bfu_N-\bfu_*-\sum_{k\in \K}\alpha_k \ica_k=0.\notag
\end{align*}
Furthermore, if $\mathcal{F}(\Rd)$ is contraction-preserving, $d\geq 2 N+2$ and there \AT{exists} some $(\mathcal{F}(\Rd), R_x)$-input $(\hat f, \hat x_0)$ such that $\GFOM_{2N+1}(\hat f, \hat x_0)$ is \emph{not} optimal for $\hat f$, then the bound~\eqref{eq:thm1} is tight.
\end{theorem}	

\begin{proof}
    The first part of the claim follows directly by establishing weak duality between~\eqref{gfom_dualPEP} and~\eqref{gfom_sdpPEP}. Indeed, \AT{one can use the following association between the constraints and dual variables along with the definition of Lagrange duality:}
		\begin{align}
		& \begin{array}{lrll}
		&\Tr {\icA_kG}+(\ica_k)^\top F+\icb_k\leq 0, &\quad\text{for all } k\in \K  &:\alpha_k,\\
		&\inner{\bg_i}{\bg_j}_G = 0, &\quad\text{for all } 0\leq j<i=1,\hdots N&:\beta_{i,j},\\
		&\inner{\bg_i}{\bx_j-\bx_0}_G = 0, &\quad\text{for all } 1\leq j \leq i=1,\hdots N\quad &:\gamma_{i,j},\\
		&\norm{\bx_0-\bx_*}_G^2-R_x^2\leq 0, & &:\tau_x.
		\end{array}\notag
		\end{align}
	
The proof for the tightness claim is presented in Appendix~\ref{sec:zeroGap}.
		\qed
\end{proof}
	
\begin{remark}
    Since \eqref{gfom_dualPEP} is an infimum problem, it has the useful property that any feasible solution to this problem corresponds to an upper bound on the worst-case accuracy of \AT{GFOM}. We take advantage of this property in the following, where bounds on the performance of \AT{GFOM} for different classes of problems are established by providing a \eqref{gfom_dualPEP}-feasible solution.
\end{remark}

\begin{remark}
    As an example of a function in $\FmuL(\Rd)$ that cannot be minimized by GFOM within $k<d$ iterations (as required by the tightness claim in the previous theorem) recall that for quadratic functions, the iterates of GFOM coincide with the iterates of the Conjugate Gradient method, hence by a well-known result, any quadratic form with $d$ distinct eigenvalues requires $d$ iterations to minimize (for a general starting point).
    In the non-smooth case, $\mathcal{C}_M(\Rd)$, one can take, for example,
    \[
        \hat f(x) = M \max(\langle x, e_1\rangle, \dots, \langle x, e_d\rangle, \|x\| - 1),
    \]
    with $x_0=0$ (where $e_i$ are the canonical unit vectors). For more details, see~\cite[Appendix A]{drori2014optimal}.
\end{remark}

To conclude this section, we note that although the analysis above was performed under Assumption~\ref{a:ic}, for classes of functions for which a set of interpolation conditions is either unknown or complex, the analysis can still proceed using a set of \emph{necessary} conditions for interpolability, with the only change being that tightness claims no longer apply.

\section{The subspace-search elimination procedure}\label{sec:gfom_to_fsfom}
	In this section, we introduce a technique for constructing first-order methods with a worst-case absolute inaccuracy that is guaranteed to be not worse than that of \AT{GFOM}.
	We begin by stating the main technical result, we then outline the SSEP technique, and finally demonstrate the application of the technique on several cases.
	
\begin{theorem}
    Let $N\in \mathbb{N}$, $R_x\geq 0$, $\mathcal{F}$ a class of \ccp functions for which Assumption~\ref{a:ic} holds, and let
	$(\{\tilde \alpha_k\}$, $\{\tilde \beta_{i,j}\}$, $\{\tilde\gamma_{i,j}\}, \tilde \tau_x)$
	be a feasible solution to~\eqref{gfom_dualPEP} that attains the \AT{objective value $\tilde \omega$.} 
	For any $(\mathcal{F}, R_x)$-input $(f, x_0)$,
	if $\{x_i\}$ is a sequence that satisfies
	\begin{equation}
	\inner{f'(x_i)}{\sum_{j=0}^{i-1}\tilde\beta_{i,j} f'(x_j) + \sum_{j=1}^{i} \tilde\gamma_{i,j}(x_j-x_0)}=0, \quad i=1,\hdots,N, \label{eq:newalgo_cstr}
	\end{equation}
	then the bound $f(x_N) - \fopt \leq \tilde\omega$ holds
	for any choice of $f'(x_i)\in \partial f(x_i)$.
	\label{thm:gfom_reconstruction}
\end{theorem}
		
The proof, presented in Appendix~\ref{sec:SSEP}, is based on the observation that by carefully aggregating the constraints in~\eqref{gfom_dualPEP}, it is possible to reach a PEP for methods satisfying~\eqref{eq:newalgo_cstr} without adversely affecting the optimal value of the PEP.
	
	Based on this result, the design procedure can be summarized as follows.
	\begin{oframed}
		\textbf{Subspace-search elimination procedure (SSEP)}
		\begin{enumerate}
			\item Choose $N\geq 0$, $R_x>0$ and a set of interpolation conditions $\{(\icA_{k}, \ica_k, \icb_k)\}_{k\in \K}$ for $\mathcal{F}$.
			\item Find a feasible solution $(\{\tilde \alpha_k\}, \{\tilde \beta_{i,j}\},$ $\{\tilde\gamma_{i,j}\}, \tilde \tau_x)$ to~\eqref{gfom_dualPEP}, and denote the objective value of the solution by $\tilde \omega$.
			\item Find a method satisfying~\eqref{eq:newalgo_cstr}.
			\item[$\rightarrow$] The worst-case absolute inaccuracy at its $N$th iterate is guaranteed to be at most $\tilde \omega$.
		\end{enumerate}
	\end{oframed}
	
\begin{remark}\label{rem:ssep_bound_discuss}
	Theorem~\ref{thm:gfom_reconstruction} states that any point feasible to~\eqref{gfom_dualPEP} can be used as a basis for constructing new methods with worst-case performances that are bounded by the value of the objective at that feasible point. In particular, this applies to an optimal solution of~\eqref{gfom_dualPEP}, hence in cases where~\eqref{gfom_dualPEP} attains the worst-case performance of \AT{GFOM} (e.g., under the tightness conditions in Theorem~\ref{thm:dual_pep_for_gfom}), the worst-case performance of any method constructed according to~\eqref{eq:newalgo_cstr} from an optimal solution is guaranteed to be equal to or better than the worst-case performance of \AT{GFOM}.
\end{remark}
	
Equality~\eqref{eq:newalgo_cstr} can be enforced in different ways. Perhaps the most straightforward one by an appropriate fixed-step size policy as detailed below.

\begin{corollary}\label{cor:ssepfixedstep}
Let $N\in \mathbb{N}$, $R_x\geq 0$, $\mathcal{F}$ a class of \ccp functions for which Assumption~\ref{a:ic} holds, and let
$(\{\tilde \alpha_k\}$, $\{\tilde \beta_{i,j}\}$, $\{\tilde\gamma_{i,j}\}, \tilde \tau_x)$
be a feasible solution to~\eqref{gfom_dualPEP} that attains the \AT{objective value $\tilde \omega$}  and satisfies $\tilde\gamma_{i,i}\neq 0$ for $i=1,\dots,N$.
For any $(\mathcal{F}, R_x)$-input $(f, x_0)$,
if $\{x_i\}$ is defined by
\begin{align}
		x_i=x_0 -\sum_{j=1}^{i-1} \frac{\tilde\gamma_{i,j}}{\tilde\gamma_{i,i}}(x_j-x_0)-\sum_{j=0}^{i-1}\frac{\tilde\beta_{i,j}}{\tilde\gamma_{i,i}} f'(x_j), \quad i=1,\hdots,N,\label{eq:reconstructed_algorithm}
\end{align}
then the bound $f(x_N) - \fopt \leq \tilde\omega$ holds
for any choice of $f'(x_i)\in \partial f(x_i)$.
\end{corollary}

\begin{proof}
	Under the assumption that $\tilde\gamma_{i,i}\neq 0$, any sequence satisfying~\eqref{eq:reconstructed_algorithm} also satisfies~\eqref{eq:newalgo_cstr} and hence Theorem~\ref{thm:gfom_reconstruction} directly applies.\qed
\end{proof}
	
	\begin{remark}
		Since the optimization variables $\gamma_{i,i}$ in~\eqref{gfom_sdpPEP} are the dual variables to the constraints  ``$\inner{g_i}{x_i-x_0}=0$'' that define how \AT{GFOM} takes advantage of the first-order information at $x_i$,
		the condition $\tilde\gamma_{i,i}\neq 0$ in the previous corollary appears to be a natural requirement in this setting, and
		indeed, this condition is fulfilled in all cases treated in the sequel.
		See also~\cite[Theorem~3]{Article:Drori} for a similar condition that arises in a related context.
	\end{remark}

	In addition to the fixed-step method defined in Corollary~\ref{cor:ssepfixedstep}, there are additional strategies for enforcing~\eqref{eq:newalgo_cstr}, where, in particular, the iterates of \AT{GFOM} satisfy these equalities. As described in the examples below,	this flexibility allows the construction of methods that have additional properties, such as independence on problem parameters (e.g., Lipschitz constants or initial distance to optimality $R_x$).

	\subsection{Example: non-smooth convex minimization}\label{ex:ns_gfom}

	Consider the problem of minimizing a non-smooth convex function
	\[\min_{x\in\mathbb{R}^d} f(x),\]
	with $f\in\mathcal{C}_{M}$ (i.e., $f$ is convex with $\|f'(x)\|\leq M$ for all $x\in \Rd$ and for all $f'(x)\in \partial f(x)$) and under the assumption that the distance between the initial point and an optimal point $\norm{x_0 - \xopt}$ is bounded by a constant $R_x$.

	We start the SSEP by establishing a worst-case bound on GFOM in this case.
	As discussed above, for this purpose it is sufficient to find a dual feasible solution to~\eqref{gfom_dPEP} when the interpolation conditions used matches the class of functions under consideration.
	
\begin{lemma}\label{lem:nonsmoothfeasible}
	The following assignment is feasible to~\eqref{gfom_dualPEP} under the interpolation conditions for~$\mathcal{C}_{M}$ defined in Example~\ref{ex:nonsmooth_ic},
	and attain the objective value of $\frac{MR_x}{\sqrt{N+1}}$:
	\begin{equation}\label{eq:dual_feas_nsmooth}
	\begin{aligned}
	& \alpha_i=\frac{R_x}{2M(N+1)^{3/2}}, \quad i=0,\hdots,N,\\
	&   \begin{aligned}
		& \alpha_{i-1,i}=\frac{i}{N+1}, && i=1,\hdots,N;
		&& \alpha_{*,i}=\frac{1}{N+1}, && i=0,\hdots,N, \\
		& \gamma_{i,i}=\frac{i+1}{N+1}, && i=1,\hdots,N;
		&& \gamma_{i,i-1}=\frac{-i}{N+1}, && i=2,\hdots,N,
		\end{aligned}\\
	& \beta_{i,j}=\frac{R_x}{M(N+1)^{3/2}}, \quad i=1,\hdots,N,\ j=0,\hdots,i-1,\\
	& \tau_x=\frac{M}{2R_x\sqrt{N+1}}, \\
	\end{aligned}
	\end{equation}
	where all the other optimization variables in~\eqref{gfom_dualPEP} are set to zero.
\end{lemma}
For the sake of coherence, the following proof relies on the SDP formalism developed above. One can note, though, that it is possible to reformulate the proof below using equivalent sum-of-squares arguments.
	\begin{proof}
		Since the inequality constraint in \eqref{gfom_dualPEP} clearly holds for \eqref{eq:dual_feas_nsmooth}, it is enough to verify the positive-semidefinite constraint and the equality constraint, i.e.:
		\begin{align*}
		&\begin{aligned}
		\sum_{i=0}^N\alpha_i \icA_i+\sum_{i\neq j\in \I}\alpha_{i,j} \icA_{i,j}&+\sum_{i=1}^N \bg_i\odot \left[\sum_{j=0}^{i-1}\beta_{i,j}\bg_j+\sum_{j=1}^i\gamma_{i,j}(\bx_j-\bx_0) \right] +\tau_x[(\bx_0-\bx_*)\odot(\bx_0-\bx_*)]\succeq 0,
		\end{aligned} \\
		&\bfu_N-\sum_{i=0}^N\alpha_i \ica_i-\sum_{i\neq j\in \I}\alpha_{i,j} \ica_{i,j}=0.
		\end{align*}
		Substituting with \eqref{eq:dual_feas_nsmooth} and the definition of $\icA$ and $\ica$, we reach
		\begin{align*}
		& \frac{R_x}{2M(N+1)^{3/2}} \sum_{i=0}^N \bg_i\odot \bg_i + \sum_{i=1}^N \frac{i}{N+1} \bg_i\odot (\bx_{i-1}-\bx_i) + \frac{1}{N+1} \sum_{i=0}^N \bg_i\odot (\bx_*-\bx_i) \\&\qquad + \sum_{i=1}^N \bg_i\odot \left[ \frac{R_x}{M(N+1)^{3/2}} \sum_{j=0}^{i-1}\bg_j+\frac{-i}{N+1}(\bx_{i-1}-\bx_0)+\frac{i+1}{N+1}(\bx_i-\bx_0) \right] \\&\qquad +\frac{M}{2R_x\sqrt{N+1}}[(\bx_0-\bx_*)\odot(\bx_0-\bx_*)]\succeq 0, \\
		& \bfu_N-\frac{1}{N+1}\sum_{i=0}^N(\bfu_i-\bfu_*)-\sum_{i=1}^N \frac{i}{N+1} (\bfu_i-\bfu_{i-1})=0.
		\end{align*}
		
		The first condition can be verified by showing that it is equal to the following positive-semidefinite rank-one matrix
		(this may require some work to obtain directly, but can easily be verified by developing both expressions): 
		\[
		\frac{M}{2R_x\sqrt{N+1}}\left(\bx_0-\bx_*-\frac{1}{\sqrt{N+1}}\frac{R_x}{M}\sum_{i=0}^{N}\bg_i\right)\odot\left(\bx_0-\bx_*-\frac{1}{\sqrt{N+1}}\frac{R_x}{M}\sum_{i=0}^{N}\bg_i\right).
		\]
		The \AT{equality} is straightforward to verify.	Finally, the objective value is given by
		\begin{align*}
		& \tau_x R_x^2-\sum_{k\in \K} \alpha_k \icb_k = \frac{M}{2R_x\sqrt{N+1}} R_x^2 +\sum_{i=0}^N \frac{R_x}{2M(N+1)^{3/2}} M^2 = \frac{M R_x}{\sqrt{N+1}}.
		\end{align*}\qed
	\end{proof}
	
	The following is now immediate from Theorem~\ref{thm:dual_pep_for_gfom}, Proposition~\ref{prop:nonsmoothcp} and Theorem~\ref{thm:gfom_reconstruction}.
\begin{corollary}\label{cor:gfom_subgrad}
Let $(f, x_0)$ be a $(\mathcal{C}_M, R_x)$-input for some $M, R_x\geq 0$.
\begin{enumerate}
\item For any $N\in \mathbb{N}$
    \[ 
	    f(\GFOM_N(f, x_0))-\fopt\leq \frac{M R_x}{\sqrt{N+1}}.
    \]
    Furthermore, this bound is tight when $d\geq 2N+2$.
\item For any sequence $x_1,\dots,x_N$ that satisfies
    \begin{equation}\label{eq:rec_fs_gfom_nonsmooth}
        \inner{f'(x_i)}{x_i-\left[\frac{i}{i+1}x_{i-1}+\frac{1}{i+1}x_0-\frac{1}{i+1}\frac{R_x}{M\sqrt{N+1}}\sum_{j=0}^{i-1}f'(x_j)\right]}=0, \quad i=1,\hdots,N,
    \end{equation}
we have
\[ 
    f(x_N)-\fopt\leq \frac{M R_x}{\sqrt{N+1}}.
\]
\end{enumerate}
\end{corollary}
	
	As noted above, sequences satisfying \eqref{eq:rec_fs_gfom_nonsmooth} can be generated in several ways. We start by describing an efficient implementation of the fixed-step scheme described in Corollary~\ref{cor:ssepfixedstep}.
	\begin{oframed}
		\textbf{SSEP-based subgradient method}
		\begin{itemize}
			\item[] Input: $f\in\mathcal{C}_{M}(\mathbb{R}^d)$, initial guess $x_0\in\mathbb{R}^d$ such that $\norm{x_0-x_*}\leq R_x$, number of iterations~$N$. \medskip
			\item[] For $i=1,\hdots,N$:
			\begin{align*}
			&y_{i}=\frac{i}{i+1}x_{i-1}+\frac{1}{i+1}x_0\\
			&d_i= \frac{1}{i+1} \sum_{j=0}^{i-1} f'(x_j) \\
			&x_{i}=y_i - \frac{R_x}{M\sqrt{N+1}} d_i
			\end{align*}
			\item[] Output: $x_N$.
		\end{itemize}
	\end{oframed}
	
\begin{corollary}
Let $(f, x_0)$ be a $(\mathcal{C}_M, R_x)$-input for some $M, R_x\geq 0$.	If $x_N$ is an output of the SSEP-based subgradient method given $f$, $x_0$ and $N\in \mathbb{N}$, then
\[ 
    f(x_N)-\fopt\leq \frac{MR_x}{\sqrt{N+1}}.
\]
\end{corollary}
\begin{proof}
	The claim then follows directly from the second part of Corollary~\ref{cor:gfom_subgrad}.\qed
\end{proof}
	
	Note though that the SSEP-based subgradient method has a guarantee on its \emph{last} iterate, whereas the guarantees on standard subgradient methods are usually either on the averaged iterate $f\left(\frac{1}{N+1}\sum_{i=0}^Nx_i\right)-\fopt$ or in terms of the best iterate $\min_{0\leq i\leq N} f(x_i)-\fopt$. Interestingly, the SSEP-based subgradient method appears to be similar to the so-called quasi-monotone subgradient methods~\cite{nesterov2015quasi}.
	
	As noted above, there are additional ways of enforcing equality~\eqref{eq:newalgo_cstr}, allowing the introduction of methods with different properties.
	As an example, one can subsume prior knowledge of the constants $R_x$, $M$ and $N$ by using an exact line search procedure, as demonstrated by the following optimal subgradient method.
	\begin{oframed}
		\textbf{SSEP-based subgradient method with an exact line search}
		\begin{itemize}
			\item[] Input: $f\in\mathcal{C}_{M}(\mathbb{R}^d)$, initial guess $x_0\in\mathbb{R}^d$. \medskip
			\item[] For $i=1,2,\hdots$:
			\begin{align*}
			&y_{i}=\frac{i}{i+1}x_{i-1}+\frac{1}{i+1}x_0\\
			&d_i= \frac{1}{i+1} \sum_{j=0}^{i-1} f'(x_j) \\
			&\alpha\in\argmin_{\alpha\in\real} f(y_i-\alpha d_i)\\
			&x_{i}=y_i-\alpha d_i \\
			&\text{Choose }f'(x_i)\in\partial f(x_i)\text{ such that } \inner{f'(x_i)}{d_{i}}=0
			\end{align*}
		\end{itemize}  
	\end{oframed}
	
\begin{corollary}\label{cor:nsmooth_ELS}
Let $(f, x_0)$ be an $(\mathcal{C}_M, R_x)$-input for some $M, R_x\geq 0$.
For any sequence $\{x_i\}$ generated by SSEP-based subgradient method with an exact line search given $f$ and $x_0$
\[ 
    f(x_N)-\fopt\leq \frac{MR_x}{\sqrt{N+1}}, \quad \forall N\in \mathbb{N}.
\]
\end{corollary}
	\begin{proof} 
		First, note that from the first-order optimality conditions on the exact line search procedure, for all $i$ there exist $f'(x_i)\in\partial f(x_i)$ that satisfies the requirement $\inner{f'(x_i)}{d_{i}}=0$.
		Now, from definition of $x_i$ and $y_i$, we have
		\begin{align*}
		& \inner{f'(x_i)}{x_i-\left[\frac{i}{i+1}x_{i-1}+\frac{1}{i+1}x_0-\frac{1}{i+1}\frac{R_x}{M\sqrt{N+1}}\sum_{j=0}^{i-1}f'(x_j)\right]}\\
		& =\inner{f'(x_i)}{x_{i}-y_i+\frac{R_x}{M\sqrt{N+1}} d_i} =\inner{f'(x_i)}{-\alpha d_i+\frac{R_x}{M\sqrt{N+1}} d_i} =0,
		\end{align*} 
		which concludes the proof, since this establishes~\eqref{eq:rec_fs_gfom_nonsmooth} as required by Corollary~\ref{cor:gfom_subgrad}.
		\qed
	\end{proof}
	
	\begin{remark}
	In the setting considered in this section, it is known that no first-order method can have a worst-case absolute inaccuracy behavior that better than 
	\begin{equation}
	f(x_N)-\fopt\geq \frac{MR_x}{\sqrt{N+1}},\label{eq:lb_nsmooth}
	\end{equation}
	after $N$ iterations when $d\geq N+2$~\cite[Theorem 2]{drori2014optimal}, hence the subgradient algorithms described above are optimal and the corresponding bounds are tight.
    \end{remark}

	Note that the methods developed above are not the first subgradient schemes that attains the optimal worst-case complexity for non-smooth minimization, as it is  achieved, for example by the optimal step-size policy $\frac{R_x}{M\sqrt{N+1}}$ proposed in~\cite[Section 3.2.3]{Book:Nesterov} and by the Kelley-like cutting plane method developed in~\cite{drori2014optimal}.

	\subsection{Example: smooth convex minimization}\label{ex:gfom_ogm}
	In this section we demonstrate the application of SSEP on the problem of minimizing a smooth convex function. We show that \AT{GFOM} attains the best possible worst-case bound on this problem, and that the resulting fixed-step method is the optimized gradient method (OGM) developed in~\cite{Article:Drori,kim2014optimized}.
	Finally, we construct a method that has the same worst-case performance, but does not require prior knowledge on the problem parameters, at the cost of performing a one-dimensional line search at each iteration.

	We use the following notations:
	\begin{align*}
	\theta_{0,N} &:= 1, \\
	\theta_{i+1,N} &:= \frac{1+\sqrt{4\theta_{i,N}^2+1}}{2},\quad i=0,1,\dots,N-1,\\
	\theta_{N,N} &:= \frac{1+\sqrt{8\theta_{N-1,N}^2+1}}{2}.
	\end{align*}

\begin{lemma}\label{lem:smoothfeasible}
The following assignment is feasible to~\eqref{gfom_dualPEP} under the interpolation conditions for~$\mathcal{F}_{0,L}$ defined in Example~\ref{ex:strongly_convex_ic}, and \AT{attains} the objective value of $\frac{LR_x^2}{2\theta_{N,N}^2}$:
\begin{alignat*}{3}
	&\alpha_{i-1,i}=\frac{2\theta_{i-1,N}^2}{\theta_{N,N}^2},\ \ & i=1,\dots,N;\quad
	&\alpha_{*,i}=\frac{2\theta_{i,N}}{\theta_{N,N}^2},\ \ & i=0,\dots,N-1;\quad
	& \alpha_{*,N}=\frac{1}{\theta_{N,N}}, \\
	& \gamma_{i,i-1}=\frac{-2\theta_{i-1, N}^2}{\theta_{N,N}^2},\ \ & i=2,\dots,N; \quad
	& \gamma_{i,i}=\frac{2\theta_{i,N}^2}{\theta_{N,N}^2},\ \ & i=1,\dots,N-1;\quad
	&\gamma_{N,N}=1, \\
	&\tau_x=\frac{L}{2\theta_{N,N}^2},
\end{alignat*}
and
\begin{align*}
	&\beta_{i,j}=\frac1L\frac{4\theta_{i,N}\theta_{j, N}}{\theta_{N,N}^2}, && i=2,\dots,N-1,\ j=0,\dots,i-2,\\
	&\beta_{i,i-1}=\frac1L\frac{2\theta_{i-1, N}^2+4\theta_{i-1, N}\theta_{i,N}}{\theta_{N,N}^2}, && i=1,\dots,N-1, \\
	&\beta_{N,i}=\frac1L\frac{2\theta_{i,N}}{\theta_{N,N}}, && i=0,\dots,N-2, \\
	&\beta_{N,N-1}=\frac1L\left(\frac{2\theta_{N-1, N}}{\theta_{N,N}}+\frac{2\theta_{N-1, N}^2}{\theta_{N,N}^2}\right), \\
\end{align*}
where all the other optimization variables in~\eqref{gfom_dualPEP} are set to zero.
\end{lemma}

\begin{proof}
	As in the non-smooth case, it is enough to show that the following constraints hold:
	\begin{align*}
	&\begin{aligned}
	\sum_{i\neq j\in \I}\alpha_{i,j} \icA_{i,j}&+\sum_{i=1}^N \bg_i\odot \left[\sum_{j=0}^{i-1}\beta_{i,j}\bg_j+\sum_{j=1}^i\gamma_{i,j}(\bx_j-\bx_0) \right]+
	\tau_x[(\bx_0-\bx_*)\odot(\bx_0-\bx_*)]\succeq 0,
	\end{aligned} \\
	&\bfu_N-\sum_{i\neq j\in \I}\alpha_{i,j} \ica_{i,j}=0.
	\end{align*}
	
	With some effort, one can show that the first expression above can be written as:
	\[
	\frac{L}{2\theta_{N,N}^2} \left(\bx_0-\bx_*-\frac{\theta_{N,N}}{L}\bg_N-\frac{2}{L}\sum_{i=0}^{N-1}\theta_{i,N} \bg_i\right)\odot \left(\bx_0-\bx_*-\frac{\theta_{N,N}}{L}\bg_N-\frac{2}{L}\sum_{i=0}^{N-1}\theta_{i,N} \bg_i\right),
	\]
	which is clearly a PSD matrix.
	As the second equality above also holds, it follows that the selected values define a feasible point.
	Finally, the objective value that corresponds to that point is given by
	\[
	\tau_x R_x^2 = \frac{L R_x^2}{2\theta_{N,N}^2}.
	\]
	\qed
\end{proof}
	
From Theorem~\ref{thm:dual_pep_for_gfom} Proposition~\ref{prop:smoothcp} and Theorem~\ref{thm:gfom_reconstruction} we get the following result.
\begin{corollary}\label{cor:gfom_grad}
    Let $(f, x_0)$ be an $(\FL, R_x)$-input for some $L, R_x\geq 0$.
    \begin{enumerate}
    \item For any $N\in \mathbb{N}$
    	\[ 
    		f(\GFOM_N(f,x_0))-\fopt\leq \frac{LR_x^2}{2\theta_{N,N}^2}.
    	\]
    	Furthermore, this bound is tight when $d\geq 2N+2$.
    \item For any sequence $x_1,\dots,x_N$ that satisfies
    	\begin{equation}\label{eq:rec_fs_gfom_smooth}
    	\begin{aligned}
    	    & \inner{g_i}{x_i-\left[\left(1-\frac{1}{\theta_{i,N}}\right)\left(x_{i-1}-\frac1L f'(x_{i-1})\right)+\frac{1}{\theta_{i,N}}\left(x_0-\frac{2}{L}\sum_{j=0}^{i-1}\theta_{j, N} f'(x_j)\right)\right]}=0,\quad i=1,\hdots,N,
    	\end{aligned} 
    	\end{equation}
    	we have
    	\[ 
    	    f(x_N)-\fopt\leq \frac{LR_x^2}{2\theta_{N,N}^2}.
        \]
    \end{enumerate}
\end{corollary}
	
As an immediate application of Corollary~\ref{cor:ssepfixedstep}, we recover the optimized gradient method (OGM), which was developed in~\cite{Article:Drori,kim2014optimized}.
\begin{oframed}
	\textbf{Optimized gradient method (OGM)~\cite{kim2014optimized}}
	\begin{itemize}
		\item[] Input: $f\in\mathcal{F}_{0,L}(\mathbb{R}^d)$, initial guess $x_0\in\mathbb{R}^d$, number of iterations $N$. \medskip
		\item[] For $i=1,\hdots,N$:
		\begin{align*}
		y_i&=\left(1-\frac{1}{\theta_{i,N}}\right)x_{i-1}+\frac{1}{\theta_{i,N}}x_0\\
		d_i&=\left(1-\frac{1}{\theta_{i,N}}\right)f'(x_{i-1})+\frac{2}{\theta_{i,N}}\sum_{j=0}^{i-1} \theta_{j, N} f'(x_j)\\
		x_{i}&=y_i- \frac{1}{L} d_i
		\end{align*}
		\item[] Output: $x_N$.
	\end{itemize}
\end{oframed}
	
As in the non-smooth case, one can trade the knowledge of the problem parameters with an exact line search procedure, resulting in an optimized gradient method with exact line search.
\newpage
	\begin{oframed}
		\textbf{Optimized gradient method with exact line search (OGM-LS)}
		\begin{itemize}
			\item[] Input: $f\in\mathcal{F}_{0,L}(\mathbb{R}^d)$, initial guess $x_0\in\mathbb{R}^d$, number of iterations $N$. \medskip
			\item[] For $i=1,\hdots,N$: 
			\begin{align*}
			y_i&=\left(1-\frac{1}{\theta_{i,N}}\right)x_{i-1}+\frac{1}{\theta_{i,N}}x_0\\
			d_i&=\left(1-\frac{1}{\theta_{i,N}}\right)f'(x_{i-1})+\frac{2}{\theta_{i,N}}\sum_{j=0}^{i-1}\theta_{j, N} f'(x_j)\\
			\alpha&\in\argmin_{\alpha\in\real} f(y_i-\alpha d_i)\\
			x_{i}&=y_i-\alpha d_i
			\end{align*}
			\item[] Output: $x_N$.
		\end{itemize}
	\end{oframed}
	
	We omit the rate of convergence proofs, as they are identical to the proofs presented in the previous section.
	As in the non-smooth case, tightness of the worst-case bounds can be established by observing that they coincide with the lower complexity bound for the problem~\cite[Corollary 4]{drori2016exact} and therefore they cannot be improved in the large-scale setting ($d\geq N+2$).

	\begin{remark}
		Before proceeding, let us note that the results above are reminiscent to the historical developments premising accelerated methods. Indeed, acceleration was first discovered in the work of Nemirovski and Yudin~\cite{nemirovski1982orth,nemirovski1983information}, who needed two or three-dimensional subspace-searches for obtaining the optimal convergence rate for smooth convex minimization. This rather strong requirement  was removed later on in the work of Nesterov, resulting in the first version of the celebrated fast gradient method~\cite{Nesterov:1983wy} not requiring any line (or space) search.
	\end{remark}

	\subsection{Example: a universal method for non-smooth and smooth convex minimization}
	In this short section, we build upon the results of Section~\ref{ex:ns_gfom} and Section~\ref{ex:gfom_ogm} to develop a \emph{universal} method for both smooth and non-smooth minimization, i.e., a method that does not require any knowledge on which of the two classes the function belongs to, nor does it require knowledge on the specific parameters of the classes. The knowledge is replaced by the capability of performing exact three-dimensional subspace minimizations.
\newpage
	\begin{oframed}
		\textbf{A universal method for non-smooth and smooth minimization (UM)}
		\begin{itemize}
			\item[] Input: $f\in\mathcal{F}_{0,L}(\mathbb{R}^d)$ or $f\in\mathcal{C}_{M}(\mathbb{R}^d)$, initial guess $x_0\in\mathbb{R}^d$, number of iterations~$N$. \medskip
			\item[] For $i=1,\hdots,N$:
			\begin{align*}
			y^{(1)}_i &=\frac{i}{i+1}x_{i-1}+\frac{1}{i+1}x_0\\
			d^{(1)}_i  &= \frac{1}{i+1} \sum_{j=0}^{i-1} f'(x_j) \\
			y^{(2)}_i&=\left(1-\frac{1}{\theta_{i,N}}\right)x_{i-1}+\frac{1}{\theta_{i,N}}x_0\\
			d^{(2)}_i&=\left(1-\frac{1}{\theta_{i,N}}\right)f'(x_{i-1})+\frac{2}{\theta_{i,N}}\sum_{j=0}^{i-1}\theta_{j, N} f'(x_j)\\
			\alpha,\beta,\gamma&\in\argmin_{\alpha,\beta,\gamma \in\real} f(\alpha y^{(1)}_i + (1-\alpha) y^{(2)}_i - \beta d^{(1)}_i -\gamma d^{(2)}_i)\\
			x_{i}&=\alpha y^{(1)}_i + (1-\alpha) y^{(2)}_i - \beta d^{(1)}_i -\gamma d^{(2)}_i\\
			&\text{Choose }f'(x_i)\in\partial f(x_i)\text{ such that } \inner{f'(x_i)}{d_{i}}=0
			\end{align*}
			\item[] Output: $x_N$.
		\end{itemize} 
	\end{oframed}
	
\begin{corollary}\label{cor:universal_gfom}
Let $x_N$ be an output generated by UM given $f$, $x_0$ and $N\in \mathbb{N}$.
\begin{enumerate}
\item If $(f, x_0)$ is an $(\AT{\mathcal{C}_M}, R_x)$-input for some $M\geq 0$, then
	\[ 
	    f(x_N)-\fopt\leq \frac{MR_x}{\sqrt{N+1}},
	 \]
\item If $(f, x_0)$ is an $(\FL, R_x)$-input for some $L\geq 0$, then
	\[ 
	    f(x_N)-\fopt\leq \frac{LR_x^2}{2\theta_{N,N}^2}.
    \]
\end{enumerate}
\end{corollary}
	\begin{proof}
		For the first part of the claim, it is enough to establish that the sequence generated by UM satisfies~\eqref{eq:rec_fs_gfom_nonsmooth}. 
		Indeed, for $i=1,\dots,N$
		\begin{align*}
		& \inner{f'(x_i)}{x_i-\left[\frac{i}{i+1}x_{i-1}+\frac{1}{i+1}x_0-\frac{1}{i+1}\frac{R_x}{M\sqrt{N+1}}\sum_{j=0}^{i-1}f'(x_j)\right]}\\
		& = \inner{f'(x_i)}{x_i-y^{(1)}_i+\frac{R_x}{M\sqrt{N+1}}d^{(1)}_i} \\
		& = \inner{f'(x_i)}{(\alpha-1) y^{(1)}_i + (1-\alpha) y^{(2)}_i - \beta d^{(1)}_i -\gamma d^{(2)}_i+\frac{R_x}{M\sqrt{N+1}}d^{(1)}_i} \\
		& =0,
		\end{align*}
		where the last equality follows from the optimality conditions
		of the exact subspace minimization step:
		\begin{align*}
		&\inner{f'(x_i)}{y_i^{(1)}-y_i^{(2)}}=0,\\
		&\inner{f'(x_i)}{d_i^{(1)}}=0,\\
		&\inner{f'(x_i)}{d_i^{(2)}}=0.
		\end{align*}
		
		The second part of the claim follows in an analogous way using \eqref{eq:rec_fs_gfom_smooth}.\qed
	\end{proof}
	
	\section{Numerical construction of efficient methods}\label{sec:numerical_gfom_to_fsfom}
	
	In this section we focus our attention to situations where either \eqref{gfom_dualPEP} does not have a known analytical solution, or the analytical solution is too complex for practical purposes. In particular, we examine the performance of the methods generated by SSEP in the case where $f\in\mathcal{F}_{\mu,L}$ (an $L$-smooth, $\mu$-strongly convex function). Note that since $f\in \FmuL$ implies $\frac{1}{L}f\in \mathcal{F}_{\frac{\mu}{L},1}$, the worst-case analyses can be limited to the case $L=1$, where simple homogeneity arguments allow extending the results to the case of arbitrary $L>0$. For a short discussion on this topic, we refer to~\cite[Section 4.2.5]{taylor2017convex}.

	In this setting, finding analytical solutions to~\eqref{gfom_dualPEP} appears to be significantly more involved than in the previous examples, nevertheless the problem can be efficiently approximated numerically using standard SDP solvers and SSEP can proceed as described above. A problem with this approach, however, is that the computation complexity of a na\"ive implementation of Corollary~\ref{cor:ssepfixedstep} would require storing all subgradients encountered throughout the computations and performing $O(i)$ vector operations at the $i^{\text{th}}$ iteration, making such an implementation undesirable in practice. In other words, SSEP generally allows recovering a set of coefficients $\{\tilde\beta_{i,j}, \tilde\gamma_{i,j}\}$ of a first-order method
	\begin{align}
		x_i=x_0 -\sum_{j=1}^{i-1} \frac{\tilde\gamma_{i,j}}{\tilde\gamma_{i,i}}(x_j-x_0)-\sum_{j=0}^{i-1}\frac{\tilde\beta_{i,j}}{\tilde\gamma_{i,i}} f'(x_j), \quad i=1,\hdots,N,\label{eq:reconstructed_algorithm2}
    \end{align}
	that enjoys the same worst-case guarantees as GFOM; however, using such a method requires, in general, keeping track of all coefficients $\{\tilde\beta_{i,j}, \tilde\gamma_{i,j}\}$ and all previous iterates and gradients $\{x_j, f'(x_j)\}$.
	
	\sloppy An approach for overcoming those drawbacks originates from the observation that in practically all situations we encountered, the coefficients $\{\tilde\beta_{i,j}, \tilde\gamma_{i,j}\}$ resulting from an optimal solution of~\eqref{gfom_dualPEP} enjoyed advantageous structural properties allowing to \emph{factor} the parameters $\{\tilde\beta_{i,j}, \tilde\gamma_{i,j}\}$ in a way that (a) does not require storing all coefficients (those coefficients are essentially \emph{separable} in $i$ and $j$), and (b) does not require storing all previous iterates and gradients in memory.
	Below we present one such algorithm structure, and demonstrate how its parameters can be extracted from the numerical solution of~\eqref{gfom_dualPEP}.
	
	\paragraph{Factorization of the SSEP-based method for strongly convex minimization}
	As discussed above, one can observe that optimal solutions to~\eqref{gfom_dualPEP} often enjoys certain advantageous structural properties. 	In particular, in the case of smooth strongly convex optimization, instances of the SSEP method fit the form~\eqref{eq:algo_eff_form2} below within high accuracy levels.
	Note that the form~\eqref{eq:algo_eff_form2} is a straightforward generalization of OGM, which is the SSEP method derived for the non-strongly convex case, $\mu=0$, established in Section~\ref{ex:gfom_ogm} (see e.g.~\cite[Section 7.1]{kim2014optimized}).
	
\begin{oframed}
		\textbf{SSEP-based gradient method for smooth strongly convex minimization}
		\begin{itemize}
			\item[] Input: $f\in\mathcal{F}_{\mu,L}(\mathbb{R}^d)$, initial guess $y_0=x_0\in\mathbb{R}^d$, maximum number of iterations $N$. \medskip
			\item[] For $i=1,\hdots,N$:
			\begin{equation}
			\begin{aligned}\label{eq:algo_eff_form2}
			& y_{i}=x_{i-1}-\frac1Lf'(x_{i-1})\\
			& x_{i}=y_{i}+\zeta_i (y_i-y_{i-1})+\eta_i(y_i-x_{i-1})
			\end{aligned}
			\end{equation}
			\item[] Output: $x_N$.
		\end{itemize}  
\end{oframed}

	Let us quickly describe a procedure for deriving candidate values for $\{(\eta_i,\zeta_i)\}$ to be used in~\eqref{eq:algo_eff_form2} given a set $\{\tilde\beta_{i,j}, \tilde\gamma_{i,j}\}$ of numerically-derived coefficients for~\eqref{eq:reconstructed_algorithm2}. We start the procedure by recursively eliminating all instances of $x_j$ with $j\geq 1$, from the right-hand-side of algorithm~\eqref{eq:reconstructed_algorithm2} using their definition, reaching a canonical form
	\[
	x_i = x_0 -\sum_{j=0}^{i-1}h_{i,j} f'(x_j).
	\]
	This representation is important as all fixed-step first-order methods have a unique representation using $\{h_{i,j}\}$, but generally no unique representation in terms of $\{\tilde\beta_{i,j}, \tilde\gamma_{i,j}\}$. Then, one can write algorithm~\eqref{eq:algo_eff_form2} in the same format with step-sizes $ \{h_{i,j}'\}$, which in this case should satisfy
    \begin{equation}\label{eq:factor}
	\begin{aligned}
	    & h_{-1,j}'=h_{0,j}'=0 \text{ for all $j$}\\
	    & h_{i,j}'=\left\{\begin{array}{cc} 
	        \tfrac1L (1+\zeta_i+\eta_i) & \text{ if } j=i-1 \\
	        h_{i-1,j}'(1+\zeta_i)-\tfrac1L \zeta_i& \text{ if } j=i-2  \\
	        h_{i-1,j}'+\zeta_i(h_{i-1,j}'-h_{i-2,j}') & \text{ otherwise.}
	    \end{array}\right\} \ \text{ for }i\geq 1.
	    \end{aligned}
	\end{equation}
	Now, assuming that both implementations~\eqref{eq:reconstructed_algorithm2} and~\eqref{eq:algo_eff_form2} describe the same algorithm, we have, in particular, $h_{i,i-1}=h_{i,i-1}'$ and $h_{i,i-2}=h_{i,i-2}'$. As a result, one can identify candidate values for $\eta_i$ and $\zeta_i$ as follows:
	\begin{equation*}
	    \begin{aligned}
	    \zeta_i&=\tfrac{h_{i,i-2}-h_{i-1,i-2}}{h_{i-1,i-2}-1/L}\\
	    \eta_i&=L\, h_{i,i-1}-1-\zeta_i,
	    \end{aligned}
	\end{equation*}
	where we also arbitrarily set $\zeta_0=0$. Finally, one can numerically verify that $h_{i,j}\approx h_{i,j}'$ indeed holds for all $i,j$ by generating the full set $\{h_{i,j}'\}$ using~\eqref{eq:factor} and the candidate values $\{(\eta_i,\zeta_i)\}$. 
	
	The numerical values of $\{(\eta_i,\zeta_i)\}$ derived in the case of $N=10$ for a few values of $\kappa$'s are reported in Table~\ref{tab:smoothstrcvx_gfom}. Since the approach is limited by our capability to accurately solve SDPs, it is important to note that PEPs can be used again for validating the performances of the final method (shown in the one before last row in Table~\ref{tab:smoothstrcvx_gfom}).
	
	\input{SmoothStrCvx_gfom_table.tex}

\paragraph{Comparison to other methods}
	Figure~\ref{Fig:OGMSC_comp} presents a comparison of the worst-case bounds in the case $\kappa=100$ for GFOM, for an SSEP-based method performing no line searches, the celebrated fast (or accelerated) gradient method (FGM) for smooth strongly convex minimization~\cite[Theorem 2.1.12]{Book:Nesterov}, and the very recent triple momentum method (TMM)~\cite{VanScoy2017}.
	These worst-case bounds were derived numerically by solving the corresponding PEPs using the interpolation conditions presented in Example~\ref{ex:strongly_convex_ic} (see~\cite{Article:Drori,taylor2015smooth} for details on the derivation of PEPs for fixed-step methods). Note that the bound for the SSEP method was generated for the form~\eqref{eq:reconstructed_algorithm2}; bounds for the efficient form~\eqref{eq:algo_eff_form2} behave almost exactly like the bounds for the form~\eqref{eq:reconstructed_algorithm2}---the difference could not be observed from this plot---and were therefore omitted from the comparison.
	
	Finally, the numerical results presented above support the conjectures that the subspace-searching GFOM enjoys an $O( (1-\sqrt{2}\sqrt{\kappa^{-1}})^N R_x^2)$ rate of convergence while the SSEP-based method that does not perform any line searches, converges at the faster rate $O( (1-2\sqrt{\kappa^{-1}})^N R_x^2)$ (matching the rate of convergence bounds on TMM~\cite[Theorem~1]{VanScoy2017}).
	This is not in contraction with the theory, as noted in Remark~\ref{rem:ssep_bound_discuss}, however, we are currently unable to find an intuitive explanation to this phenomenon besides the algebraic observation that the PEPs corresponding to the methods~\eqref{eq:reconstructed_algorithm2} and~\eqref{eq:algo_eff_form2} have less degrees of freedom than the PEP for GFOM.

\begin{figure}
	\begin{center}
		\subfigure{
			\begin{tikzpicture}
			\begin{axis}[
			ymode=log,  
			xlabel=Iteration counter $i$,
			ylabel=Worst-case of $f(x_i) - \fopt$]
			\addplot[mark=square,red,thick] table [y=TMM, x=N,]{strongly_convex_comparison_bounded_distance.txt};
			\addlegendentry{TMM}
			\addplot[mark=o,teal,thick] table [y=FGM, x=N]{strongly_convex_comparison_bounded_distance.txt};
			\addlegendentry{FGM}
			\addplot[mark=triangle,blue,thick] table [y=GFOMAdrien, x=N]{strongly_convex_comparison_bounded_distance.txt};
			\addlegendentry{GFOM}
			\addplot[mark=diamond,cyan,thick] table [y=SSEP, x=N, skip coords between index={24}{25}]{strongly_convex_comparison_bounded_distance.txt};
			\addlegendentry{SSEP}
			\end{axis}
			\end{tikzpicture}
		}
		\caption{Numerical approximation of the worst-case performance of the FGM, TMM \cite{VanScoy2017}, GFOM and the SSEP method with no line-search for minimizing a $1$-smooth $0.01$-strongly convex function for instances with $\|x_0-\xopt\|\leq 1$.
		For the fast gradient method, we used the inertial parameter ${(1-\sqrt{\mu/L})}/{(1+\sqrt{\mu/L})}$.
			}
		\label{Fig:OGMSC_comp}
	\end{center}
\end{figure}
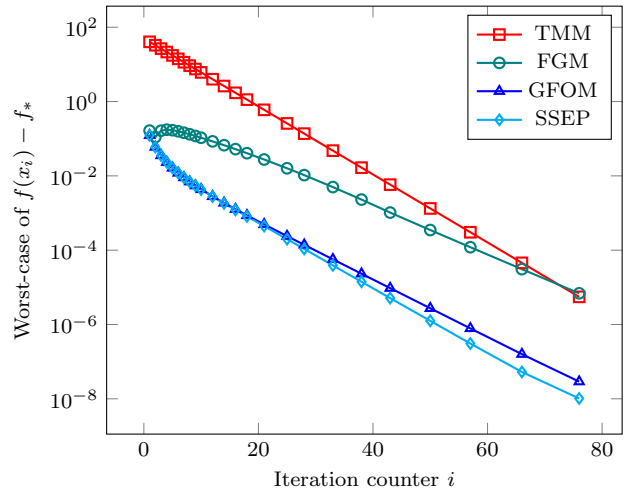

	\section{Conclusion}\label{sec:ccl}
	The main goal of this work is to provide a systematic and efficient approach for designing first-order optimization algorithms. The contribution is essentially threefold: first, we extend the performance estimation framework for obtaining worst-case guarantees for a greedy method that performs arbitrary subspace searches, and show that the generated guarantees are tight in the large-scale setting under some weak assumptions. Then, we describe a methodology for systematically designing fixed-step methods that share the same worst-case guarantees as the subspace searching greedy method. Finally, based on the new methodology, we derive optimal methods for non-smooth and smooth convex minimization and versions of these methods that do not require prior knowledge of the problem parameters.

	As illustrated in \AT{Section~\ref{sec:numerical_gfom_to_fsfom}}, the methodology can also serve in cases where numerical results cannot easily be converted to practical analytical optimization schemes. For example, for real-time embedded optimization~\cite{diehl2009efficient}, where it is acceptable to spend some time performing pre-computations in order to more efficiently perform simple repetitive routines.

	Direct extensions of the approach include considering additional families of functions and oracles, such as composite functions involving proximal terms and projections~\cite{beck2009fast,nesterov2007gradient}, inexactness~\cite{de2016worst,devolder2013intermediate,devolder2014first,schmidt2011convergence}, stochastic oracles arising in finite sums~\cite{defazio2014saga,johnson2013accelerating,roux2012stochastic}, or block-coordinate descent~\cite{nesterov2012efficiency,wright2015coordinate}.
	Even further extensions include considering additional variants of the greedy method, e.g., using alternative optimality criteria, as the distance to an optimal point. 
	
	\paragraph{Implementation} All the worst-case performance analyses presented above were numerically validated using the \textsc{pesto} toolbox~\cite{taylorperformance}. The code for reproducing the worst-case guarantees is available at \begin{center}
		\url{https://github.com/AdrienTaylor/GreedyMethods}
	\end{center}
	Numerical experiments were produced using \textsc{cvx} and \textsc{yalmip}~\cite{grant2008cvx,Article:Yalmip} along with MOSEK~\cite{Article:Mosek}.
	
\newpage
\appendix

\section{Proof of Lemma~\ref{lem:pep_for_gfom}}\label{sec:GFOMproof}
	We start the proof of Lemma~\ref{lem:pep_for_gfom} with the following a technical lemma.
	
	\begin{lemma} \label{lem:lindep} Let $\mathcal{F}$ be a class of contraction-preserving \ccp functions (see Definition~\ref{def:contraction}), and let $S=\{(x_i,g_i,f_i)\}_{i\in \I}$ be an $\mathcal{F}$-interpolable set satisfying
		\begin{align}
		& \inner{g_i}{g_j}=0, \quad\text{for all } 0\leq j<i=1,\hdots,N,\label{eq:opt_cond1}\\
		& \inner{g_i}{x_j-x_0}=0,\quad \text{for all } 1\leq j\leq i=1,\hdots,N,\label{eq:opt_cond2}
		\end{align}
		then there exists $\{\hat x_i\}_{i\in \I}\subset \Rd$ such that the set $\hat S=\{(\hat x_i,g_i,f_i)\}_{i\in \I}$ is $\mathcal{F}$-interpolable, and
		\begin{align}
		& \norm{\hat x_0 - \hat x_*}\leq \norm{x_0-x_*},  \label{eq:lindep_contract}\\
		&\hat x_i \in \hat x_0 + \mathrm{span}\{g_0,\hdots,g_{i-1}\},\quad {i=0,\hdots, N}. \label{eq:lindep_span}
		\end{align}
	\end{lemma}
	
	\begin{proof}
		By the orthogonal decomposition theorem there exists $\{h_{i,j}\}_{0\leq j<i\leq N} \subset \real$ and $\{v_i\}_{0\leq i\leq N} \subset \Rd$ 
		with $\inner{g_k}{v_i}=0$ for all $0\leq k<i \leq N$ such that
		\begin{align*}
		x_i&=x_0-\sum_{j=0}^{i-1} h_{i,j}g_j +v_i, \quad { i=0,\hdots, N},
		\end{align*}
		furthermore, there exist $r_*\in \Rd$ satisfying $\inner{r_*}{v_j}=0$ for all $0\leq j \leq N$ and some $\{\nu_{j}\}_{0\leq j\leq N}\subset \real$, such that
		\[
		x_*=x_0 + \sum_{j=0}^N \nu_{j}v_j + r_*.
		\]
		By \eqref{eq:opt_cond1} and~\eqref{eq:opt_cond2} it then follows that for all $k\geq i$
		\[
		\inner{g_k}{v_i} = \inner{g_k}{x_i-x_0+\sum_{j=0}^{i-1} h_{i,j} g_j} = 0,
		\]
		hence, together with the definition of $v_i$, we get
		\begin{equation}\label{eq:lindep_orthogonality}
		\inner{g_k}{v_i}=0, \quad {i,k=0,\hdots,N}.
		\end{equation}
		
		Let us now choose $\{\hat x_i\}_{i\in \I}$ as follows:
		\begin{align*}
		& \hat x_0:=x_0,\\
		& \hat x_i:=x_0-\sum_{j=0}^{i-1} h_{i,j} g_j, \quad { i =0,\hdots, N}, \\
		& \hat x_* := x_0+r_*.
		\end{align*}
		It follows immediately from this definition that~\eqref{eq:lindep_span} holds, it thus remains to show that $\hat S$ is $\mathcal{F}$-interpolable and that~\eqref{eq:lindep_contract} holds.
		
		In order to establish that $\hat S$ is $\mathcal{F}$-interpolable, from Definition~\ref{def:contraction} it is enough to show that the conditions in \eqref{eq:contracted_xi} are satisfied.
		This is indeed the case, as $\inner{g_j}{\hat x_i - \hat x_0}=\inner{g_j}{x_i-x_0}$
		follows directly from definition of $\{\hat x_i\}$ and~\eqref{eq:lindep_orthogonality},
		whereas $\norm{\hat x_i - \hat x_j}\leq \norm{x_i-x_j}$ 
		in the case $i,j\neq *$ follows from
		\begin{align*}
		\norm{x_i-x_j}^2&=\norm{x_0-\sum_{k=0}^{i-1} h_{i,k} g_k+v_i-x_0+\sum_{k=0}^{j-1} h_{j,k} g_k-v_j}^2\\
		&=\norm{\hat x_i - \hat x_j}^2+\norm{v_i-v_j}^2\\
		&\geq \norm{\hat x_i - \hat x_j}^2,  \quad {i,j=0,\hdots, N},
		\end{align*}
		and in the case $j=*$, follows from
		\begin{align*}
		\norm{x_i-x_*}^2&=\norm{x_0-\sum_{k=0}^{i-1} h_{i,k} g_k+v_i-x_0 -\sum_{j=0}^N \nu_{j}v_j - r_*}^2\\
		&=\norm{\hat x_i - \hat x_*}^2+\norm{v_i-\sum_{j=0}^N \nu_{j}v_j}^2\\
		&\geq \norm{\hat x_i - \hat x_*}^2, \quad {i=0,\hdots, N},
		\end{align*}
		where for the second equality we used $\inner{v_i}{r_*}=0$.
		The last inequality also establishes~\eqref{eq:lindep_contract}, which completes the proof.
		\qed
	\end{proof}

	\paragraph{Proof of Lemma~\ref{lem:pep_for_gfom}} 
	By the first-order necessary and sufficient optimality conditions (see e.g.,~\cite[Theorem 3.5]{ruszczynski2006nonlinear}), the definitions of $x_i$ and $f'(x_i)$ in~\eqref{E:gfomstep} and~\eqref{E:gfom_grad_selection}
	can be equivalently defined as a solution to the problem of finding $x_i\in \Rd$ and $f'(x_i)\in\partial f(x_i)$ {($0\leq i\leq N$)}, that satisfy:
	\begin{align*}
	& \inner{f'(x_i)}{f'(x_j)}=0, \quad\text{for all } 0\leq j<i=1,\hdots,N,  \notag \\
	& x_i\in x_0+\mathrm{span}\{f'(x_0),\hdots,f'(x_{i-1})\}, \quad\text{for all } i=1,\hdots,N,  \notag \\
	\end{align*}
	hence the problem~\eqref{Intro:PEP} can be equivalently expressed as follows:
	
	\begin{align}
	\sup_{ f, \left\{x_i\right\}_{i \in \I}, \{f'(x_i)\}_{i\in \I}}& f(x_N)-\fopt \label{gfompep:functional} \\
	\text{{subject to:} }
	& f\in\mathcal{F}(\Rd),\ x_* \text{ is a minimizer of } f, \notag\\
	& f'(x_i) \in \partial f(x_i), \quad\text{for all } i\in \I, \notag \\ 
	& \norm{x_0-x_*}\leq R_x, \notag\\
	& \inner{f'(x_i)}{f'(x_j)}=0, \quad\text{for all } 0\leq j<i=1,\hdots,N,  \notag \\
	& x_i\in x_0+\mathrm{span}\{f'(x_0),\hdots,f'(x_{i-1})\}, \quad\text{for all } i=1,\hdots,N.  \notag
	\end{align}
	
	Now, since all constraints in \eqref{gfompep:functional} depend only on the first-order information of $f$ at $\{x_i\}_{i\in \I}$, by taking advantage of Definition~\ref{def:Finterpolability} we can denote $f_i:=f(x_i)$ and $g_i:=f'(x_i)$ and treat these and as optimization variables, thereby reaching the following \emph{equivalent} formulation
	\begin{align}
	\sup_{\{(x_i,g_i,f_i)\}_{i\in \I}}&\ f_N-f_* \label{gfompep:discrete}\\
	\text{{subject to:} }&  \{(x_i,g_i,f_i)\}_{i\in \I} \text{ is }\mathcal{F}(\Rd)\text{-interpolable},   \notag\\
	& \norm{x_0-x_*}\leq R_x, \notag\\
	& g_*=0, \notag \\
	&\inner{g_i}{g_j}= 0, \ \text{for all } 0\leq j<i=1,\hdots N,\notag\\
	& x_i\in x_0+\mathrm{span}\{g_0,\hdots,g_{i-1}\},\quad\text{for all } i=1,\hdots,N.\notag
	\end{align}
	Since~\eqref{gfom_dPEP} is a relaxation of~\eqref{gfompep:discrete}, we get
	\[
	f(x_N) - \fopt \leq \val\eqref{Intro:PEP} \leq \val\eqref{gfom_dPEP},
	\]
	which establishes the bound~\eqref{eq:thm1}.
	
	\medskip
	In order to establish the second part of the claim, let $\varepsilon>0$. We will proceed to show that there exists some valid input for GFOM $(f, x_0)$, such that $f(\GFOM_N(f, x_0)) - \fopt \geq \val \eqref{gfom_dPEP}-\varepsilon$.
	
	Indeed, by the definition of \eqref{gfom_dPEP}, there exists a set $S=\{(x_i,g_i,f_i)\}_{i\in \I}$ that satisfies the constraints in~\eqref{gfom_dPEP} and reaches an objective value $f_N-f_* \geq \val \eqref{gfom_dPEP}-\varepsilon$.
	Since $S$ satisfies the requirements of Lemma~\ref{lem:lindep} (as these requirements are constraints in~\eqref{gfom_dPEP}), there exists a set of vectors $\{\hat x_i\}_{i\in \I}$ for which
	\begin{align*}
	& \norm{\hat x_0- \hat x_*}\leq R_x, \\
	& \hat x_i\in \hat x_0 + \mathrm{span}\{g_0,\hdots,g_{i-1}\},\quad i=0,\dots,N,
	\end{align*} 
	hold, and in addition, $\hat S:=\{(\hat x_i,g_i,f_i)\}_{i\in \I}$ is
	$\mathcal{F}(\Rd)$-interpolable.
	By definition of an $\mathcal{F}(\Rd)$-interpolable set, it follows that there exists a function $\hat f\in \mathcal{F}(\Rd)$ such that
	$\hat f(\hat x_i) = f_i$, $g_i \in \partial \hat f(\hat x_i)$, hence satisfying
	\begin{align*}
	& \inner{\hat f'(\hat x_i)}{\hat f'(\hat x_j)} = 0, \quad\text{for all } 0\leq j<i=1,\hdots,N, \\
	& \hat x_i\in \hat x_0+\mathrm{span}\{\hat f'(x_0),\hdots, \hat f'(\hat x_{i-1})\}, \quad\text{for all } i=1,\hdots,N.
	\end{align*}
	Furthermore, since $g_*=0$ we have that $\hat x_*$ is an optimal solution of $\hat f$.
	
	We conclude that the sequence $\hat x_0, \dots, \hat x_N$ forms a valid execution of \AT{GFOM} on the input $(\hat f, \hat x_0)$, that the requirement $\norm{\hat x_0 - \hat x_*}\leq R_x$ is satisfied, and that the output of the method, $\hat x_N$, attains the absolute inaccuracy value of $\hat f(\hat x_N) -\hat f(\hat x_*) = f_N - f_* \geq \val \eqref{gfom_dPEP}-\varepsilon$.
	
	\qed

\section{Proof of Theorem~\ref{thm:dual_pep_for_gfom}}\label{sec:zeroGap}
	\begin{lemma} Suppose there exists a pair $(f,x_0)$ such that $f\in\mathcal{F}$, $\norm{x_0-x_*}\leq R_x$ and 
	$\GFOM_{2N+1}(f, x_0)$ is not optimal for $f$,
		then~\eqref{gfom_sdpPEP} satisfies Slater's condition. In particular, no duality gap occurs between the primal-dual pair~\eqref{gfom_sdpPEP}, \eqref{gfom_dualPEP}, and the dual optimal value is attained.\label{lem:zerodualitygap}
	\end{lemma}
	\begin{proof}

		Let $(f,x_0)$ be a pair satisfying the premise of the lemma and denote by $\{x_i\}_{i\geq 0}$ the sequence generated according to GFOM and by $\{f'(x_i)\}_{i\geq 0}$ the subgradients chosen at each iteration of the method, respectively.
		By the assumption that the optimal value is not obtained after $2N+1$ iterations, we have $f(x_{2N+1})>\fopt$.

		We show that the set $\{(\tilde x_i,\tilde g_i, \tilde f_i)\}_{i\in \I}$ with 
		\begin{align*}
		& \tilde x_i:=x_{2i}, \quad i=0,\hdots,N, \\
		& \tilde x_*:=x_*, \\
		& \tilde g_i:=f'(x_{2i}), \quad i=0,\hdots,N, \\
		& \tilde g_*:=0, \\
		& \tilde f_i:=f(x_{2i}), \quad i=0,\hdots,N, \\
		& \tilde f_*:=f(\xopt),
		\end{align*} corresponds to a Slater point for~\eqref{gfom_sdpPEP}.
		
		In order to proceed, we consider the Gram matrix $\tilde G$ and the vector $\tilde F$ constructed from the set $\{(\tilde x_i, \tilde g_i, \tilde f_i)\}_{i\in \I}$ as in Section~\ref{sec:primalpep}. We then continue in two steps:
		\begin{itemize*}
			\item[(i)] we show that $(\tilde G, \tilde F)$ is feasible for~\eqref{gfom_sdpPEP},
			\item[(ii)] we show that $\tilde G\succ 0$.
		\end{itemize*}
		The proofs follow.
		\begin{itemize}
			\item[(i)] 
			First, we note that the set $\{(\tilde x_i, \tilde g_i, \tilde f_i)\}_{i\in \I}$ satisfies the interpolation conditions for $\mathcal{F}$, as it was obtained by taking the values and gradients of a function in $\mathcal{F}$.
			Furthermore, since $\tilde x_0 = x_0$ and $\tilde x_*=x_*$ we also get that the initial condition $\norm{\tilde x_0-\tilde x_*}\leq R_x$ is respected, and since $\{x_i\}$ correspond to the iterates of \AT{GFOM}, we also have by Lemma~\ref{lem:lindep} that
			\begin{align*}
			&\inner{\tilde g_i}{\tilde g_j}= 0, \quad \text{for all } 0\leq j<i=1,\hdots N, \\
			& \inner{\tilde g_i}{\tilde x_j-\tilde x_0}= 0, \quad \text{for all } 1\leq j \leq i=1,\hdots N.
			\end{align*}
			It then follows from the construction of
			$\tilde G$ and $\tilde F$ and by~\eqref{eq:sdp_vardef} that $\tilde G$ and $\tilde F$ satisfies the constrains of~\eqref{gfom_sdpPEP}.
			\item[(ii)] In order to establish that $\tilde G\succ 0$ it suffices to show that the vectors 
			\[
			\{\tilde g_0,\hdots, \tilde g_N ; \tilde x_1- \tilde x_0,\hdots,\tilde x_N- \tilde x_0 ; \tilde x_*- \tilde x_0 \}
			\]
			are linearly independent. Indeed, this follows from Lemma~\ref{lem:lindep}, since these vectors are all non-zero, and since $\tilde x_*$ does not fall in the linear space spanned by $\tilde g_0,\hdots, \tilde g_N ; \tilde x_1- \tilde x_0,\hdots, \tilde x_N- \tilde x_0$ (as otherwise $x_{2N+1}$ would be an optimal solution).
		\end{itemize}
		We conclude that $(\tilde G, \tilde F)$ forms a Slater point for~\eqref{gfom_sdpPEP}.\qed
	\end{proof}

\paragraph{Proof of Theorem~\ref{thm:dual_pep_for_gfom}}
The bound follows directly from
\[
    f(\GFOM_{N}(f, x_0)) - \fopt \leq \val \eqref{gfom_dPEP} \leq \val\eqref{gfom_sdpPEP},
\]
established by Lemmas~\ref{lem:pep_for_gfom} and~\ref{lem:pep_sdp_relax}.
The tightness claim follows from the tightness claims of Lemmas~\ref{lem:pep_for_gfom}, \ref{lem:pep_sdp_relax} and~\ref{lem:zerodualitygap}.
\qed
	
\section{Proof of Theorem~\ref{thm:gfom_reconstruction}}\label{sec:SSEP}
	We begin the proof of Theorem~\ref{thm:gfom_reconstruction} by recalling a well-known lemma on constraint aggregation, showing that it is possible to aggregate the constraints of a minimization problem while keeping the optimal value of the resulting program bounded from below.
	\begin{lemma}\label{lem:aggregation}
		Consider the problem 
		\begin{equation}\tag{P}\label{E:lemmaprimal}
		w:=\min\{f(x): h(x)=0, g(x)\leq 0\},
		\end{equation}
		where $f:\Rd\rightarrow \real$,  $h:\Rd\rightarrow \real^n$, $g:\Rd\rightarrow \real^m$
		are some (not necessarily convex) functions, and suppose $(\tilde \alpha, \tilde \beta)\in \real^{n}\times \real_+^{m}$ is a feasible point for the Lagrangian dual of \eqref{E:lemmaprimal} that attains the value $\tilde \omega$. Let $k\in \mathbb{N}$, and let $M\in \real^{n \times k}$ be a linear map such that $\tilde \alpha \in \mathrm{range}(M)$, then 
		\begin{equation}\tag{P$'$}\label{E:lemmaaggragated}
		w':=\min\{f(x): M^\top h(x)=0, g(x)\leq 0\}
		\end{equation}
		is bounded from below by $\tilde \omega$.
	\end{lemma}
	\begin{proof}
		Let 
		\[
		L(x, \alpha, \beta) = f(x)+\alpha^\top h(x) + \beta^\top g(x)
		\]
		be the Lagrangian for the problem \eqref{E:lemmaprimal}, then by the assumption on $(\tilde \alpha, \tilde \beta)$ we have
		$
		\min_x L(x, \tilde \alpha, \tilde \beta) = \tilde \omega.
		$
		Now, 
		let $u\in \real^k$ be some vector such that $Mu = \tilde\alpha$, then
		for every $x$ in the domain of \eqref{E:lemmaaggragated} 
		\begin{align*}
		& \tilde\alpha^\top h(x)  = u^\top M^\top h(x) = 0, \\
		& \tilde \beta^\top g(x)\leq 0,
		\end{align*}
		where that last inequality follows from nonnegativity of $\tilde \beta$.
		We get
		\[
		f(x) \geq f(x) + \tilde\alpha^\top h(x) + \tilde \beta^\top g(x) = L(x, \tilde \alpha, \tilde \beta) \geq \tilde \omega, \quad \forall x: M^\top h(x)=0, g(x)\leq 0,
		\]
		and thus the desired result $w'\geq \tilde \omega$ holds.
		\qed
	\end{proof}

	Before proceeding with the proof of the main results, let us first formulate a performance estimation problem for the class of methods described by~\eqref{eq:newalgo_cstr}.
	\begin{lemma}\label{lem:fsmpep}
		Let $ R_x\geq 0$ and let $\{\beta_{i,j}\}_{1\leq i\leq N, 0\leq j\leq i-1}$, $\{\gamma_{i,j}\}_{1\leq i\leq N, 1\leq j\leq i}$ be some given sets of real numbers, 
		then for any pair $(f, x_0)$ such that $f\in\mathcal{F}(\Rd)$ and $\norm{x_0-x_*}\leq R_x$ (where $x_*\in \argmin_x f(x)$).
		Then for any sequence $\{x_i\}_{1\leq i\leq N}$ that satisfies
		\begin{equation}
		\inner{f'(x_i)}{\sum_{j=0}^{i-1}\beta_{i,j} f'(x_j) + \sum_{j=1}^{i} \gamma_{i,j}(x_j-x_0)}=0, \quad i=1,\hdots,N
		\end{equation}
		for some $f'(x_i)\in \partial f(x_i)$, the following bound holds:
		\begin{align*}
		f(x_N)-\fopt \leq &\sup_{ F\in\mathbb{R}^{N+1}, G\in\mathbb{R}^{2N+2\times 2N+2}} F^\top \bfu_N - F^\top \bfu_* \\
		& \begin{array}{lrl}
		\text{{subject to:} } 
		&\Tr {\icA_kG}+(\ica_k)^\top F+\icb_k\leq 0, & \quad\text{for all } k\in \K,\\
		&\inner{\bg_i}{\sum_{j=0}^{i-1} \beta_{i,j}\bg_j + \sum_{j=1}^{i} \gamma_{i,j}(\bx_j-\bx_0)}_G = 0, &\quad\text{for all } i=1,\hdots N,\\
		&\norm{\bx_0-\bx_*}_G^2-R_x^2\leq 0, &\\
		& G\succeq 0.
		\end{array}
		\end{align*}
	\end{lemma}
	We omit the proof since it follows the exact same lines as for \eqref{gfom_sdpPEP} (c.f.\ the derivations in~\cite{Article:Drori,taylor2015smooth}).

	\paragraph{Proof of Theorem~\ref{thm:gfom_reconstruction}}
	The key observation underlying the proof is that by taking the PEP for GFOM~\eqref{gfom_sdpPEP} and aggregating the constraints that define its iterates, we can reach a PEP for the class of methods~\eqref{eq:newalgo_cstr}.
	Furthermore, by Lemma~\ref{lem:aggregation}, this aggregation can be done in a way that maintains the optimal value of the program, thereby reaching a specific method in this class whose corresponding PEP attains an optimal value that is at least as good as that of the PEP for \AT{GFOM}.

	We perform the aggregation of the constraints as follows:
	for all $i=1,\dots,N$ we 
	aggregate the constraints which correspond to $\{\beta_{i,j}\}_{0\leq j<i}$, $\{\gamma_{i,j}\}_{1\leq j\leq i}$ (weighted by $\{\tilde \beta_{i,j}\}_{0\leq j<i}$, $\{\tilde\gamma_{i,j}\}_{1\leq j\leq i}$, respectively)
	into a single constraint, reaching
	\begin{align}
	w'(N, \mathcal{F}(\mathbb{R}^d),R_f,R_x,R_g) := &\sup_{ F\in\mathbb{R}^{N+1}, G\in\mathbb{R}^{2N+2\times 2N+2}} F^\top \bfu_N - F^\top \bfu_* \tag{PEP-SSEP$(\mathcal{F}(\Rd))$}\label{ssep_dPEP}\\
	& \begin{array}{lrl}
	\text{{subject to:} } 
	&\Tr {\icA_kG}+(\ica_k)^\top F+\icb_k\leq 0, & \quad\text{for all } k\in \K,\\
	&\inner{\bg_i}{\sum_{j=0}^{i-1}\tilde \beta_{i,j}\bg_j + \sum_{j=1}^{i} \tilde \gamma_{i,j}(\bx_j-\bx_0)}_G= 0, &\quad\text{for all } i=1,\hdots N,\\
	&\norm{\bx_0-\bx_*}_G^2-R_x^2\leq 0, &\\
	& G\succeq 0.
	\end{array}\notag
	\end{align}
	
	By Lemma~\ref{lem:aggregation} and the choice of weights $\{\tilde \beta_{i,j}\}_{0\leq j<i}$, $\{\tilde\gamma_{i,j}\}_{1\leq j\leq i}$ it follows that
	\[
	w'(N, \mathcal{F}(\mathbb{R}^d),R_x) \leq \tilde\omega.
	\]
	Finally, by Lemma~\ref{lem:fsmpep}, we conclude that $w'(N, \mathcal{F}(\mathbb{R}^d),R_x)$ forms an upper bound on the performance of the method~\eqref{eq:newalgo_cstr}, i.e.,
	for any valid pair $(f, x_0)$ and any $\{x_i\}_{i\geq 0}$ that satisfies \eqref{eq:newalgo_cstr} we have
	\[
	f(x_N)-\fopt \leq w'(N, \mathcal{F}(\mathbb{R}^d),R_x)\leq \tilde\omega.
	\]
	\qed

	\bibliographystyle{spmpsci}      
	\bibliography{bib_}{}   
	
\end{document}

%% file: SmoothStrCvx_gfom_table.tex
\begin{table}[t]
\begin{center}
{\renewcommand{\arraystretch}{1.4}
\begin{tabular}{@{}c|cc|cc|cc|cc@{}}
\specialrule{2pt}{1pt}{1pt}
 &\multicolumn{2}{c|}{ $\kappa=\infty$} & \multicolumn{2}{c|}{$\kappa=1000$} & \multicolumn{2}{c|}{$\kappa=100$} & \multicolumn{2}{c}{$\kappa=50$} \\
 $i$  & $\zeta_i$ & $\eta_i$ & $\zeta_i$ & $\eta_i$ & $\zeta_i$ & $\eta_i$ & $\zeta_i$ & $\eta_i$ \\
\hline
1 & $0$             &  $0.6180$  & $0$           &  $0.6173$        &  $0$      & $0.6110$   & $0$         & $0.6039$\\
2 & $0.2818$        &  $0.7376$  & $0.2810$      &  $0.7365$        &  $0.2744$ & $0.7259$   & $0.2671$    & $0.7142$\\
3 & $0.4340$        &  $0.7977$  & $0.4325$      &  $0.7960$        &  $0.4184$ & $0.7812$   & $0.4030$    & $0.7650$\\
4 & $0.5311$        &  $0.8346$  & $0.5286$      &  $0.8324$        &  $0.5068$ & $0.8132$   & $0.4835$    & $0.7929$\\
5 & $0.5988$        &  $0.8597$  & $0.5954$      &  $0.8570$        &  $0.5661$ & $0.8339$   & $0.5352$    & $0.8099$\\
6 & $0.6489$        &  $0.8780$  & $0.6448$      &  $0.8750$        &  $0.6085$ & $0.8485$   & $0.5708$    & $0.8216$\\
7 & $0.6876$        &  $0.8920$  & $0.6829$      &  $0.8888$        &  $0.6413$ & $0.8607$   & $0.5981$    & $0.8321$\\
8 & $0.7185$        &  $0.9030$  & $0.7136$      &  $0.9001$        &  $0.6701$ & $0.8738$   & $0.6240$    & $0.8462$\\
9 & $0.7437$        &  $0.9120$  & $0.7392$      &  $0.9097$        &  $0.6985$ & $0.8892$   & $0.6539$    & $0.8663$\\
10& $0.5542$        &  $0.6663$  & $0.5514$      &  $0.6652$        &  $0.5265$ & $0.6553$   & $0.4988$    & $0.6441$\\
\hline
$f(\GFOM_{10}(f,x_0))-f_*\leq$ &\multicolumn{2}{c|}{$\frac{L\normsq{x_0-x_*}}{159.07}$} & \multicolumn{2}{c|}{$\frac{L\normsq{x_0-x_*}}{164.95}$}& \multicolumn{2}{c|}{$\frac{L\normsq{x_0-x_*}}{230.87}$}& \multicolumn{2}{c}{$\frac{L\normsq{x_0-x_*}}{340.41}$}  \\
Using~\eqref{eq:reconstructed_algorithm2}, $f(x_{10})-f_*\leq$ &\multicolumn{2}{c|}{$\frac{L\normsq{x_0-x_*}}{159.07}$} & \multicolumn{2}{c|}{$\frac{L\normsq{x_0-x_*}}{165.04}$}& \multicolumn{2}{c|}{$\frac{L\normsq{x_0-x_*}}{232.86}$}& \multicolumn{2}{c}{$\frac{L\normsq{x_0-x_*}}{347.88}$}  \\
Using~\eqref{eq:algo_eff_form2}, $f(x_{10})-f_*\leq$ &\multicolumn{2}{c|}{$\frac{L\normsq{x_0-x_*}}{159.07}$} & \multicolumn{2}{c|}{$\frac{L\normsq{x_0-x_*}}{165.04}$}& \multicolumn{2}{c|}{$\frac{L\normsq{x_0-x_*}}{232.86}$}& \multicolumn{2}{c}{$\frac{L\normsq{x_0-x_*}}{347.88}$}  \\
$\max_{i,j} \lvert h_{i,j}-{h}_{i,j}'\rvert \approx$ & \multicolumn{2}{c|}{ $8.2\times 10^{-9}$} & \multicolumn{2}{c|}{ $2.1\times 10^{-6}$} & \multicolumn{2}{c|}{ $1.8\times 10^{-6}$} & \multicolumn{2}{c}{ $9.1 \times 10^{-6}$} \\ 
\specialrule{2pt}{1pt}{1pt}
      
\end{tabular}       
}
\end{center}
\caption{Parameters for the implementation~\eqref{eq:algo_eff_form2} of the SSEP-based gradient methods for smooth strongly convex minimization with $N=10$ and $\kappa\in\{\infty, 1000, 100, 50\}$ recovered using MOSEK~\cite{Article:Mosek}, and the associated worst-case guarantees on $f(x_{10})-f(x_*)$. The worst-case objective function accuracy for GFOM and for the vanilla (unfactored) SSEP method~\eqref{eq:reconstructed_algorithm2} are also given for validation purposes. Finally, we provide the absolute inaccuracy observed between the coefficients $h_{i,j}$ of the vanilla SSEP method~\eqref{eq:reconstructed_algorithm2} versus the corresponding coefficients $h_{i,j}'$ obtained for the factored method~\eqref{eq:algo_eff_form2}.
}
\label{tab:smoothstrcvx_gfom}
\end{table}